\documentclass[reqno]{amsart}
\usepackage{textcomp}

\usepackage[dvipsnames]{xcolor}
\usepackage{amsmath}
\usepackage{amsthm}
\usepackage{hyperref}
\usepackage{amsfonts,graphics,amsthm,amsfonts,amscd,latexsym}
\usepackage{tikz-cd}
\usepackage{epsfig}
\usepackage{flafter}
\usepackage{longtable}
\usepackage{mathtools}
\usepackage{comment}
\usepackage{stmaryrd}
\usepackage{enumitem}

\usepackage{mathabx,epsfig}

\hypersetup{
    colorlinks=true,    
    linkcolor=blue,          
    citecolor=blue,      
    filecolor=blue,      
    urlcolor=blue           
}
\usepackage{tikz}
\usetikzlibrary{graphs,positioning,arrows,shapes.misc,decorations.pathmorphing}

\tikzset{
    >=stealth,
    every picture/.style={thick},
    graphs/every graph/.style={empty nodes},
}

\tikzstyle{vertex}=[
    draw,
    circle,
    fill=black,
    inner sep=1pt,
    minimum width=5pt,
]
\usepackage[position=top]{subfig}
\usepackage{amssymb}
\usepackage{color}

\calclayout

\usetikzlibrary{decorations.pathmorphing}
\tikzstyle{printersafe}=[decoration={snake,amplitude=0pt}]

\newcommand{\Cl}{\operatorname{Cl}}
\newcommand{\Pic}{\operatorname{Pic}}

\newcommand{\pp}{\mathbb{P}}

\newcommand{\qq}{\mathbb{Q}}
\newcommand{\zz}{\mathbb{Z}}

\newcommand{\rr}{\mathbb{R}}

\newcommand{\oo}{\mathcal{O}}

\def\O#1.{\mathcal {O}_{#1}}   
\def\pr #1.{\mathbb P^{#1}}    
\def\af #1.{\mathbb A^{#1}}   
\def\ses#1.#2.#3.{0\to #1\to #2\to #3 \to 0} 
\def\xrar#1.{\xrightarrow{#1}}   
\def\K#1.{K_{#1}}      
\def\bA#1.{\mathbf{A}_{#1}}   
\def\bM#1.{\mathbf{M}_{#1}}    
\def\bL#1.{\mathbf{L}_{#1}}    
\def\bB#1.{\mathbf{B}_{#1}}    
\def\bK#1.{\mathbf{K}_{#1}}   
\def\subs#1.{_{#1}}     
\def\sups#1.{^{#1}}

\DeclareMathOperator{\mld}{mld}

\usepackage{tikz}
\usetikzlibrary{matrix,arrows,decorations.pathmorphing}

  \newtheorem{theorem}{Theorem}[section]
  \newtheorem{lemma}[theorem]{Lemma}
  \newtheorem{proposition}[theorem]{Proposition}
  \newtheorem{corollary}[theorem]{Corollary}

  \newtheorem{definition}[theorem]{Definition}
  \newtheorem{example}[theorem]{Example}

  \newtheorem{question}[theorem]{Question}

\newtheorem{remark}[theorem]{Remark}

\theoremstyle{remark}

\numberwithin{equation}{section}

\usepackage[all]{xy}

\begin{document}

\title[Degenerations of cluster type varieties]
{Degenerations of cluster type varieties}

\author[J.~Moraga]{Joaqu\'in Moraga}
\address{UCLA Mathematics Department, Box 951555, Los Angeles, CA 90095-1555, USA
}
\email{jmoraga@math.ucla.edu}

\author[J.P.~Z\'u\~niga]{Juan Pablo Z\'u\~niga}
\address{Facultad de Matem\'aticas,
Pontificia Universidad Cat\'olica de Chile, Santiago, Chile.
}
\email{jpzuniga3@uc.cl}


\thanks{JM was partially supported by NSF research grant DMS-2443425. JPZ was supported by the ANID National Doctoral Scholarship 2022–21221224.}

\begin{abstract}
We study degenerations
of cluster type varieties and pairs. 
Our first theorem proves that degenerations 
of toric pairs are finite quotients of toric pairs.
In a similar vein, under some mild conditions, 
we prove that degenerations of cluster type pairs
are finite quotients of cluster type pairs.
Then, we focus on degenerations of cluster type surfaces. 
We give some general criteria for the existence of $1$-complements
on degenerations of toric surfaces. 
We prove that for {\em almost all} $(a,b,c)\in \mathbb{Z}_{\geq 1}^3$
the weighted projective plane $\mathbb{P}(a,b,c)$ has no non-trivial degenerations. 
In particular, for a Markov triple $(a,b,c)\in \zz_{\geq 2}^3$, we prove that $\pp(a^2,b^2,c^2)$ admits no non-trivial degenerations. 
Finally, we give a complete classification of the degenerations of
$\mathbb{P}(1,1,n)$ for $n\geq 3$.
\end{abstract} 

\maketitle

\setcounter{tocdepth}{1} 
\tableofcontents

\section{Introduction}

In this article, we study degenerations of varieties and pairs. 
We focus on degenerations with central fiber having klt singularities.
These are known as {\em klt degenerations}. 
The problem is particularly interesting when the general fiber $\mathcal{X}_t$ of the degeneration $\mathcal{X}\rightarrow \mathbb{D}$ is a Fano variety.
In this case, there are many possible degenerations $\mathcal{X}_0$ 
and these degenerations are interesting even from the combinatorial perspective. 
For instance, it is known that the toric degenerations of $\mathbb{P}^2$ are given by 
\begin{equation}\label{eq:markov-triple}
\mathbb{P}(a^2,b^2,c^2) \text{ with $a^2+b^2+c^2=3abc$.}
\end{equation} 
These triples $(a,b,c)\in \zz_{\geq 1}^3$ are known as {\em Markov triples}. 
The topic of degenerations of projective surfaces is classic in algebraic geometry. 
In~\cite{HP05}, Hacking and Prokhorov proved that any degeneration of $\mathbb{P}^2$ is indeed a partial smoothing of a toric surface as in~\eqref{eq:markov-triple}.
These surfaces were studied by Manetti and are nowadays known as 
{\em Manetti surfaces} (see~\cite{Man91}).
In~\cite{HP10}, Hacking and Prokhorov classify del Pezzo surfaces, with quotient singularities and Picard rank one, admitting $\mathbb{Q}$-Gorenstein smoothings. 
In~\cite{Pro19}, Prokhorov studied log canonical degenerations with Picard rank one of del Pezzo surfaces in $\qq$-Gorenstein families.
In~\cite{UZ24}, Urzua and the second author explain how to relate 
the degenerations as in~\eqref{eq:markov-triple} via birational transformations. 
Given two degenerations $\mathcal{X} \rightarrow \mathbb{D}$ 
and $\mathcal{Y}\rightarrow \mathbb{D}$ of $\pp^2$ into 
$\mathcal{X}_0 \simeq \mathbb{P}(a,b,c)$
and $\mathcal{Y}_0\simeq \mathbb{P}(d,e,f)$, 
the authors explain how to perform small birational modifications to go from 
$\mathcal{X}\rightarrow \mathbb{D}$ to $\mathcal{Y}\rightarrow \mathbb{D}$
and how these small modifications reflect on the combinatorics of the Markov triples.
In higher dimensions, H\"oring and Peternell proved that a klt degeneration $X$ of $\mathbb{P}^n$ is indeed isomorphic to the $n$-dimensional projective space whenever $T_X$ is semistable (see~\cite{HP24}).
They also classify normal degenerations of $\pp^3$ with canonical singularities.
In~\cite{HKW25}, Hausen, Kir\'aly, and Wrobel explicitly determine all log terminal,
rational, degenerations of $\pp^2$ that admit non-trivial torus actions.
In~\cite{UZ25}, Urzua and the second author classified all Wahl singularities that appear
in the degenerations of del Pezzo surfaces of degree $d$,
extending the work of Manetti and Hacking-Prokhorov in degree $9$.
Recently in~\cite{Pen25}, Peng studied special $\mathbb{G}_m$-degenerations of del Pezzo surfaces $X$
induced by log canonical places of pairs $(X,C)$ where $C$ is a nodal curve.
Peng proved that the space of special valuations of $(X,C)$ is connected and admits a 
partition, which is locally finite, and each interval corresponds to a different $\mathbb{G}_m$-degeneration of $X$. 

Throughout this work, we focus on understanding the degenerations of cluster type
varieties and cluster type pairs.
A {\em cluster type pair} is a pair $(X,B)$ with mild singularities
for which $K_X+B\sim 0$ and there is an embedding in codimension one 
$\mathbb{G}_m^n \dashrightarrow X\setminus B$ (see Definition~\ref{def:ct}).  
A {\em cluster type variety} is a variety $X$ that admits a cluster type pair $(X,B)$ structure. 
In this case, we call $B$ a cluster type boundary. 
Cluster type varieties and pairs can be thought of as a generalization of toric varieties and pairs. 
Indeed, any toric pair is a cluster type pair.
However, the realm of cluster type pairs is much broader. 
A del Pezzo surface of degree $d\geq 2$ as well as a general del Pezzo
surface of degree $d=1$ is cluster type (see, e.g.,~\cite[Theorem 2.1 and Remark 2.2]{ALP23}). 
Thus, all the examples discussed in the introduction are cluster type varieties.
Our first aim is to understand degenerations of cluster type pairs 
for which the limit of the cluster type boundary still has reasonable singularities. 
In the next subsection, we prove that there are not many such degenerations
for toric pairs, but quite a lot for cluster type pairs.

\subsection{Degenerations of cluster type pairs}

Our first theorem states that degenerations of toric pairs
are finite quotients of toric pairs. We impose some conditions 
on the singularities of the central fiber.

\begin{theorem}\label{thm:toric-deg}
Let $\pi\colon \mathcal{X}\rightarrow \mathbb{D}$ be a projective fibration.
Let $(\mathcal{X},\mathcal{X}_0+\mathcal{B})$ be a log Calabi--Yau pair over $\mathbb{D}$ for which $(\mathcal{X},\mathcal{X}_0)$ is plt. 
If $(\mathcal{X}_t,\mathcal{B}_t)$ is a toric pair\footnote{Meaning that $\mathcal{X}_t$ is a projective toric variety and $\mathcal{B}_t$ is the reduced torus invariant boundary.} for $t \in \mathbb{D}^*$,
then $(\mathcal{X}_0,\mathcal{B}_0)$ is a finite quotient of a toric pair. 
\end{theorem}

In the previous theorem,
the boundary $\mathcal{B}_0$ on the central fiber $\mathcal{X}_0$
is defined via adjunction. 
This means that $\mathcal{B}_0$ is defined via the formula.
\[
K_{\mathcal{X}_0}+\mathcal{B}_0 = (K_{\mathcal{X}}+\mathcal{B}+\mathcal{X}_0)|_{\mathcal{X}_0}.
\]
In Example~\ref{ex:non-toric-central}, 
we show that in general $(\mathcal{X}_0,\mathcal{B}_0)$ is not a toric pair. Indeed, the central fiber $\mathcal{X}_0$ does not need to be a rational variety under the assumptions of Theorem~\ref{thm:toric-deg}.
Our next theorem gives a similar statement for degenerations of cluster type pairs.
We need to impose stronger conditions on the singularities in this case.
In the proof of Theorem~\ref{thm:toric-deg}, we will see that $(\mathcal{X},\mathcal{B})\rightarrow \mathbb{D}$ is indeed a finite quotient of an isotrivial toric family. Thus, in the toric case, there are not many interesting degenerations for which the degeneration of the toric boundary still has lc singularities.
Now, we turn to discuss the case of cluster type pairs.

\begin{theorem}\label{thm:ct-deg}
Let $\pi\colon \mathcal{X}\rightarrow \mathbb{D}$ be a projective Fano fibration.
Let $(\mathcal{X},\mathcal{X}_0+\mathcal{B})$ be a log Calabi--Yau pair of index one over $\mathbb{D}$.
Assume that $(\mathcal{X},\mathcal{X}_0)$ is plt 
and purely terminal on the complement of $\mathcal{B}$.
If $(\mathcal{X}_t,\mathcal{B}_t)$ is of cluster type for $t\in \mathbb{D}^*$, then 
$(\mathcal{X}_0,\mathcal{B}_0)$ is a finite quotient of a cluster type pair.
\end{theorem} 

We refer the reader to Definition~\ref{def:pt} for the concept of purely terminal pairs.
In contrast to the toric case, a cluster type pair may have several interesting degenerations for which the degeneration of the cluster type boundary still has lc singularities. In Example~\ref{ex:P^2-toric-model-1}, we show that some toric degenerations of $\pp^2$ can be regarded as cluster type degenerations for different embeddings of algebraic tori 
$\mathbb{G}_m^2 \hookrightarrow \pp^2\setminus C$ where $C$ is a nodal cubic.

\subsection{Degenerations of singular surfaces}

Now, we restrict ourselves to the study of degenerations of singular toric surfaces. 
One of our aims is to understand how the singularities affect the possible degenerations. 
Our third theorem states that almost every weighted projective plane has no interesting degenerations.

\begin{theorem}\label{no-deg-almost}
For almost all well-formed triples $(a,b,c)\in \zz_{\geq 1}^3$
the weighted projective plane $\mathbb{P}(a,b,c)$ has no non-trivial
$\mathbb{Q}$-Gorenstein klt degenerations. 
\label{weighteddeg}
\end{theorem}

In the previous theorem, when we write {\em almost all}, we mean that the statement holds up to a subset $S\subsetneq \zz_{\geq 1}^3$ that has density zero with the natural density endowed from $\zz^3$.
We say that a degeneration is {\em trivial} if the central fiber is isomorphic to the general fiber. 
However, in the setting of Theorem~\ref{no-deg-almost}, we will prove something stronger; the family is a product near the origin of the disk.
In order to prove Theorem~\ref{weighteddeg}, we will prove Theorem~\ref{thm:1-comp}, which is a general statement about
the existence of complements for degenerations of singular toric surfaces.
This means that, under some mild conditions, we prove that given a degeneration 
$\mathcal{X}\rightarrow \mathbb{D}$ of a singular toric surface, 
there exists some boundary $\mathcal{B}\in |-K_{\mathcal{X}}|$ for which $(\mathcal{X}_t,\mathcal{B}_t)$ has log canonical singularities for $t$ near $\{0\}\in \mathbb{D}$.
In the setting of Theorem~\ref{weighteddeg}, in most cases, we can argue that $(\mathcal{X}_t,\mathcal{B}_t)$ is toric for every $t$ and so 
the statement is similar to that of Theorem~\ref{thm:toric-deg}, which states that there are no interesting such toric degenerations.
Theorem~\ref{thm:1-comp} is rather technical and depends on some meticulous analysis of basket of singularities.
The idea of using the theory of complements to understand degenerations of del Pezzo surfaces goes back to Hacking and Prokhorov. 
The situation becomes a bit more delicate when we allow the general fiber of the degeneration to have singularities. 
We will argue that for a Markov triple $(a,b,c)\in \zz_{\geq 2}^3$ the triple 
$(a^2,b^2,c^2)$ belongs to the complement of the subset $S\subsetneq \zz_{\geq 0}^3$ of density zero mentioned above.
Thus, we conclude the following corollary.

\begin{corollary}\label{cor:markov-no-deg}
Let $(a,b,c)\in \zz_{\geq 2}^3$ be a Markov triple.
Then, the weighted projective plane $\pp(a^2,b^2,c^2)$ has no non-trivial $\qq$-Gorenstein klt degenerations.
\end{corollary} 

We note that Corollary~\ref{cor:markov-no-deg} can also be concluded from the work of Hacking and Prokhorov~\cite{HP10}. Indeed, every iterated degeneration of $\mathbb{P}(a^2,b^2,c^2)$ is indeed a degeneration of $\mathbb{P}^2$.
In upcoming work~\cite{Zun25}, the second author will prove 
some structural theorems about $\qq$-Gorenstein klt degenerations
of Hirzebruch surfaces. 
This will aim to finish the classification of $\qq$-Gorenstein
klt degenerations of minimal smooth rational surfaces. 
This motivates us to pay particular attention to the weighted projective plane 
$\mathbb{P}(1,1,n)$. 
In this direction, using the tools introduced above, we can give
a complete classification of klt degenrations of $\pp(1,1,n)$ with $n\geq 3$.

\begin{theorem}\label{thm:1-1-n}
Let $\mathcal{X}\rightarrow \mathbb{D}$ be a klt Fano degeneration of $\pp(1,1,n)$ with $n \geq 3$, 
then  for $\mathcal{X}_0$ one of the following holds:
\begin{enumerate}
\item $\mathcal{X}_0$ is a weighted projective plane, or 
\item $\mathcal{X}_0$ is a $\mathbb{G}_m$-surface which is not toric. 
\end{enumerate}
Furthermore, in the second case $\mathcal{X}_0$ is a $\qq$-Gorenstein deformation of a weighted projective plane.
\end{theorem} 

In Proposition~\ref{prop:tor-deg}, 
we give a explicit classification
of the weighted projective planes
which are $\qq$-Gorenstein klt degenerations of $\pp(1,1,n)$ with $n\geq 3$.

The paper is organized as follows.
In Section~\ref{sec:prelims}, we write some preliminary results regarding cluster type pairs, theory of complements, dual complexes, T-singularities, and Wahl singularities. 
In Section~\ref{sec:degen-cluster-type}, we prove Theorem~\ref{thm:toric-deg} and Theorem~\ref{thm:ct-deg} regarding degenerations of toric pairs as well as cluster type pairs. 
In Section~\ref{sec:complements-degen-klt-surfaces}, we prove some general statements regarding the existence of complements for degenerations of singular toric surfaces of Picard rank one. In this section, we also prove Theorem~\ref{weighteddeg} regarding the degenerations of weighted projective planes $\mathbb{P}(a,b,c)$. 
In Section~\ref{sec:degen-wps}, we classify the degenerations of weighted projective planes $\pp(1,1,n)$ with $n\geq 3$. 
Finally, in Section~\ref{sec:ex-and-quest}, we give some examples and propose some questions for further research.

\subsection*{Acknowledgements}

The authors would like to thank Audric Lebovitz, Tomoki Oda, Giancarlo Urzúa, and Jos\'e Ignacio Y\'a\~nez
for many discussions related to this article.

\section{Preliminaries}
\label{sec:prelims}

We work over the field of complex numbers. 
All the degenerations considered in this paper are over a disk $\mathbb{D}$ and are assumed to be isotrivial over $\mathbb{D}^*:=\mathbb{D}-\{0\}$ unless otherwise stated.
All our degenerations have reduced central fiber unless otherwise stated.
We say that a degeneration $\mathcal{X}\rightarrow \mathbb{D}$ is {\em trivial} if $\mathcal{X}_t\simeq \mathcal{X}_0$.
In this section, we collect several preliminary results
as well as definitions related to
singularities of pairs, 
cluster type pairs, complements, 
dual complexes, T-singularities, and Wahl singularities. 

\subsection{Singularities of pairs} 
In this subsection, we recall some basic definitions of singularities of pairs. 
For the standard definitions of singularities of pairs, for instance, klt and log canonical, we refer the reader to~\cite{KM98,Kol13}. Here, we review some concepts that may not be standard for the reader. We consider pairs $(X,B)$ with rational coefficients, i.e., $B$ is a $\qq$-divisor. Given a pair $(X,B)$ and a divisor $E$ over $X$, we write $a_E(X,B)$ for the log discrepancy of $(X,B)$ with respect to $E$. 

\begin{definition}
{\em 
A pair $(X,B)$ is said to be {\em purely log terminal} or {\em plt} for short if the pair $(X,B)$ is log canonical
and $a_E(X,B)>0$ for every divisor $E$ over $X$ which is exceptional over $X$. 
}
\end{definition}

Note that klt pairs are plt. However, there are some plt pairs which are not klt as $(\mathbb{A}^2, H_0)$ where $H_0$ is a hyperplane in $\mathbb{A}^2$. 

\begin{definition}\label{def:pt}
{\em 
A pair $(X,B)$ is said to be {\em purely terminal} if the following conditions are satisfied:
\begin{enumerate}
\item the pair $(X,B)$ is plt; and 
\item for every prime component $S\subset \lfloor B\rfloor$ the pair $(S,B_S)$, obtained from adjunction of $(X,B)$ to $S$, is terminal. 
\end{enumerate} 
}
\end{definition}

The previous definition is not standard in the literature. However, it is sensible as it plays a similar role to that of plt pairs.
Indeed, a pair $(X,B)$ is purely log terminal if and only if for every prime component $S\subset \lfloor B\rfloor$ the pair $(S,B_S)$, obtained from adjunction of $(X,B)$ to $S$, is klt\footnote{klt singularities are also known as log terminal singularities.}. 

\begin{definition}
{\em 
A {\em klt degeneration} is a projective morphism 
$\pi\colon \mathcal{X}\rightarrow \mathbb{D}$ such that the pair $(\mathcal{X},\mathcal{X}_0)$ is plt
and $\pi^*\{0\}=\mathcal{X}_0$, i.e, the central fiber is reduced.
}
\end{definition} 

If $\mathcal{X}\rightarrow \mathbb{D}$ is a klt degeneration, then the central fiber $\mathcal{X}_0$ is klt and the equality 
\[
(K_{\mathcal{X}}+\mathcal{X}_0)|_{\mathcal{X}_0}\sim_\qq K_{\mathcal{X}_0} 
\]
holds. Indeed, as $\pi^*\{0\}=\mathcal{X}_0$ the variety $\mathcal{X}$ does not have singularities of codimension two along $\mathcal{X}_0$.

\begin{definition}
{\em 
Let $X\rightarrow Z$ be a fibration. 
A {\em log Calabi--Yau pair} over $Z$ is a pair $(X,B)$ which has log canonical singularities
for which $K_X+B\sim_{\qq,Z} 0$.
We say that $(X,B)$ is a log Calabi--Yau over
}
\end{definition}

\subsection{Cluster type pairs}
In this subsection, we recall the definition of cluster type pairs and state some basic results about cluster type pairs. We refer the reader to~\cite{EFM24,AdSFM24,JM24,MY24,ELY25} for further results about cluster type pairs.

\begin{definition}\label{def:ct}
{\em 
Let $(X,B)$ be log Calabi--Yau pair of dimension $n$. 
We say that $(X,B)$ is of {\em cluster type}
or that $(X,B)$ is a {\em cluster type pair} if there exists an embedding 
in codimension one 
\[
\mathbb{G}_m^n \dashrightarrow X\setminus B, 
\]
i.e., there exists a closed subset of codimension at least two $Z\subset \mathbb{G}_m^n$ 
and an embedding $\mathbb{G}_m^n\setminus Z \hookrightarrow X\setminus B$.
}
\end{definition}

If $(X,B)$ is a cluster type pair, then we have $K_X+B\sim 0$.
Indeed, if $(X,B)$ is of cluster type, then it is crepant birational equivalent to a toric pair so $K_X+B\sim 0$ holds by~\cite{FMM25}
Furthermore, if $(X,B)$ is of cluster type, then $X\setminus B$ is covered, 
up to a subset of codimension at least two, by images of codimension one embeddings
$\mathbb{G}_m^n \dashrightarrow X\setminus B$ (see~\cite[Theorem 1.3.(3)]{EFM24}).
The following two lemmas are proved in previous works. 
See for instance~\cite[Lemma 2.25 and Lemma 2.30]{JM24}.                             

\begin{lemma}\label{lem:cluster-type-ascends}
Let $(X,B)$ be a cluster type pair. 
Let $\phi \colon Y \dashrightarrow X$ be a projective birational map
that only extracts log canonical places of $(X,B)$.
Let $(Y,B_Y)$ be the log pull-back of $(X,B)$ to $Y$.
Then, the pair $(Y,B_Y)$ is of cluster type. 
\end{lemma} 

\begin{lemma}\label{lem:cluster-type-via-dlt}
A pair $(X,B)$ is of cluster type if and only if there exists 
a dlt modification $(Y,B_Y)\rightarrow (X,B)$
and a crepant birational contraction $(Y,B_Y)\dashrightarrow (T,B_T)$
to a toric log Calabi--Yau pair.
\end{lemma}

\subsection{Theory of complements} 
In this subsection, we recall the definition of a complement.

\begin{definition}
{\em 
Let $X\rightarrow Z$ be a projective fibration. 
We say that an effective divisor $B$ on $X$ is an {\em $m$-complement over $Z$}
for $X$ over $Z$ if the following conditions are satisfied:
\begin{enumerate}
\item the pair $(X,B)$ is log canonical; and 
\item we have $m(K_X+B)\sim_Z 0$. 
\end{enumerate} 
In other words, a $m$-complement is the boundary $B$ which gives a log Calabi--Yau pair $(X,B)$ of index $m$ over $Z$, i.e., $m(K_X+B)\sim_Z 0$. 
Whenever the base $Z$ is clear from the context, we may just say that $B$ is an $m$-complement of $X$. 
If we fix a closed point $z\in Z$, we may say that $B$ is an {\em $m$-complement over $z\in Z$} if conditions (1) and (2) above hold over a neighborhood of $z\in Z$.
}
\end{definition} 

\subsection{Dual complexes} 
In this subsection, we recall the definition of dual complex and state a lemma regarding dual complexes of log Calabi--Yau pairs. 

\begin{definition}
{\em  
Let $(Y,B_Y)$ be a dlt pair with
$\lfloor B_Y\rfloor =\sum_{i\in I} B_i$ its prime decomposition.
The {\em dual complex} of $(Y,B_Y)$, denoted by $\mathcal{D}(B_Y)$, 
is defined to be $\triangle$-complex whose vertices $v_i, i\in I$ correspond
to $B_i, i\in I$ and whose $k$-cells correspond to strata of $\lfloor B_Y\rfloor$, i.e., 
connected components of $B_{i_0}\cap \dots \cap B_{i_k}$. 
}
\end{definition} 

\begin{definition}
{\em 
Let $(X,B)$ be a log Calabi--Yau pair.
The {\em dual complex} of $(X,B)$, denoted by $\mathcal{D}(X,B)$, 
is defined to be $\mathcal{D}(B_Y)$ where $(Y,B_Y)\rightarrow (X,B)$ is a dlt modification.
}
\end{definition} 

A priori, the dual complex $\mathcal{D}(X,B)$ depends on the choice of a dlt modification.
However, the simple-homotopy equivalence class of $\mathcal{D}(X,B)$ is independent
of the chosen dlt modification (see, e.g.,~\cite[Theorem 3]{dFKX}).
In particular, the PL-homotopy class of the dual complex is independent of the 
crepant birational model of the log Calabi--Yau pair.
Thus, for a cluster type pair $(X,B)$ of dimension $n$, we have 
$\mathcal{D}(X,B)\simeq_{\rm PL} \mathbb{S}^{n-1}$.
The following lemma follows from~\cite[Lemma 29 and Lemma 30]{KX16}.

\begin{lemma}\label{lem:finite-cover}
Let $(\mathcal{X},\mathcal{B})\rightarrow Z$ be a log Calabi--Yau fibration. Then, there exists a commutative diagram 
\[
\xymatrix{
(\mathcal{X},\mathcal{B})\ar[d]_-{\pi} & (\mathcal{X}',\mathcal{B}')\ar[d]^-{\pi'}\ar[l]_-{f} \\ 
Z & Z' \ar[l]_-{f_Z}
}
\]
where the following conditions are satisfied: 
\begin{enumerate} 
\item $f_Z$ is a finite Galois morphism; 
\item $f$ is a crepant finite Galois morphism; 
\item $\pi$ and $\pi'$ are log Calabi--Yau fibrations; and 
\item $\mathcal{D}(\mathcal{X}'_\eta,\mathcal{B}'_\eta)=
\mathcal{D}(\mathcal{X}'_t,\mathcal{B}'_t)$ holds for $\eta$ the generic point of $Z'$ and $t\in Z'$ a general closed point. 
\end{enumerate} 
\end{lemma}

\subsection{T-singularities and Wahl singularities} 

In this subsection, we collect some basic results about T-singularities. Additionally, we review some panoramic facts about degenerations of rational surfaces. For a more detailed exposition about T-singularities, see, for example, \cite{Urz25}. For previous work on $\mathbb{Q}$-Gorenstein degenerations of rational surfaces with klt singularities, we refer the reader to \cite{HP10},\cite{P15}, and \cite{UZ25}.

\begin{definition}
{\em Let $0<q<m$ be coprime integers. A {\em cyclic quotient singularity} $\frac{1}{m}(1,q)$ is the surface germ at $(0,0)$ of the quotient of $\mathbb{C}^2$ by $(x,y) \mapsto (\zeta x, \zeta^q y)$, where $\zeta$ is a $m$-th primitive root of $1$. 
\label{cqs}
\em} 
\end{definition}

The 2-dimensional klt singularity germs $(X,P)$ admitting a $\mathbb{Q}$-Gorenstein smoothing were classified in \cite[Proposition 3.10]{KSB88}. Singularities attaining this property are referred as T-singularities.

\begin{definition}
{\em 
A {\em T-singularity} is a Du Val singularity or a cyclic quotient singularity $\frac{1}{dn^2}(1,dna-1)$, where $d\geq 1$ and $0<a<n$ are coprime integers. Moreover, we refer to singularities of the form $\frac{1}{n^2}(1,na-1)$ as Wahl singularities.
}
\end{definition}

A normal projective surface $X$ with at most T-singularities does not necessarily admit a $\mathbb{Q}$-Gorenstein smoothing, i.e, a family $\mathcal{X}\to \mathbb{D}$ such that $\mathcal{X}_0=X$ and $\mathcal{X}_t$ is a smooth projective surface for $t\neq 0$. However, if the anticanonical divisor $-K_X$ is big, it follows that $X$ has no local to global obstructions to deform (see for example \cite[Proposition 3.1]{HP10}) and therefore is $\mathbb{Q}$-Gorenstein smoothable. Additionally, for a normal projective surface $X$ containing at most T-singularities the correction terms of the Noether formula are positive integers. Specifically, by \cite[Proposition 3.6]{HP10} we have that 
\begin{align}
K_X^2+\chi_{top}(X)+\sum_{P\in \operatorname{Sing}(X)}\mu_P=12\chi(\mathcal{O}_X).
\label{Noether}
\end{align}
Here each $\mu_P$ is given by $b_2(M_P)$, the second Betti number of the Milnor fiber $M_p$ of the smoothing of $P$. When $P$ has the form $\frac{1}{dn^2}(1,dna-1)$, it follows that $\mu_P=d-1$. Hence, when $P$ is a Wahl singularity, then $\mu_P$ vanishes. These singularities belong to the class of $\mathbb{Q}$HD (rational homology disk) singularities, whose study appeared in \cite{W76} and \cite{W81}.

We now review some topological features about $\mathbb{Q}$-Gorenstein smoothings $\mathcal{X}\to \mathbb{D}$ where the central fiber contains at most Wahl singularities. As in \cite[Section 2.2.2]{H16}, we obtain the following exact sequence in integral homology:
\begin{align}
0 \longrightarrow H_2(\mathcal{X}_t, \zz) \longrightarrow H_2(X, \zz) \longrightarrow \bigoplus_{P_i\in \operatorname{Sing(X)}}H_1(M_{P_i}, \zz) \longrightarrow H_1(\mathcal{X}_t, \zz) \longrightarrow H_1(X, \zz) \longrightarrow 0
\label{longexact}
\end{align}
where the map $H_2(\mathcal{X}_t, \zz) \longrightarrow H_2(X, \zz)$ denotes specialization of 1-cycles. In particular, when $\mathcal{X}_t$ is a rational surface, then $H_1(\mathcal{X}_t, \zz)= H_1(X, \zz)=0$. By \cite[Proposition 4.11]{K05}, it follows that $\Pic(\mathcal{X}_t)=H_2(\mathcal{X}_t,\zz)$ and $\Cl(X)=H_2(X,\zz)$. This leads to the short exact sequence 
\begin{align}
0 \longrightarrow \Pic(\mathcal{X}_t) \longrightarrow \Cl(X) \longrightarrow \bigoplus_{P_i\in \operatorname{Sing(X)}}\Cl(X,P_i)\longrightarrow 0
\label{shortexact}
\end{align}
where $\Cl(P_i\in X)\cong \zz/n_i\zz$ and $n_i$ denotes the Gorenstein index of $P_i$.

\section{Degenerations of cluster type pairs}
\label{sec:degen-cluster-type}

In this section, we prove statements about degenerations of toric pairs as well as cluster type pairs. 

\begin{proof}[Proof of Theorem~\ref{thm:toric-deg}]
Let $f\colon \mathcal{X'}\rightarrow \mathcal{X}$ be the finite cover given by Lemma~\ref{lem:finite-cover}.
By Lemma~\ref{lem:finite-cover}.(2), we have an induced 
finite cover of log Calabi--Yau pairs $(\mathcal{X}',\mathcal{B}')\rightarrow (\mathcal{X},\mathcal{B})$.
Then, up to shrinking $\mathbb{D}$ near $\{0\}$ we have a commutative diagram 
\[
\xymatrix{
(\mathcal{X},\mathcal{B})\ar[d]_-{\pi} & (\mathcal{X}',\mathcal{B}') \ar[l]_-{f}\ar[d]^-{\pi'} \\ 
\mathbb{D} & \mathbb{D}\ar[l]_-{f_\mathbb{D}}
}
\]
where $f_\mathbb{D}$ is simply given by $t\mapsto t^k$ for some suitable positive integer $k$.
The morphism $f'\colon \mathcal{X}'\rightarrow \mathbb{D}$ is a fibration by Lemma~\ref{lem:finite-cover}.(3). 
By assumption $(\mathcal{X},\mathcal{X}_0)$ is plt so by Riemann-Hurwitz we conclude that $(\mathcal{X}',\mathcal{X}_0')$ is plt as well. In particular, the variety $\mathcal{X}_0'$ is irreducible.
By Lemma~\ref{lem:finite-cover}.(4), every component of $\mathcal{B}'_t$ 
is the restriction to $\mathcal{X}'_t$ of a component of $\mathcal{B}'$. 
Note that $\pi$ and $\pi'$ have the same general log fibers; indeed $f$ is induced by a finite cover of $\mathbb{D}$ ramified over $\{0\}$. 
Therefore, the general fiber $(\mathcal{X'},\mathcal{B'})$ is a projective toric variety of dimension $n$ 
and Picard rank $\rho$. 
Shrinking $\mathbb{D}$ around $\{0\}$, we may assume that all the fibers over $\pi'$ are irreducible.
Therefore, we conclude that $\rho(\mathcal{X}'/\mathbb{D})\leq \rho$.
On the other hand, as every component of $\mathcal{B}'_t$ is the restriction to $\mathcal{X}'_t$ of a component of $\mathcal{B}$, 
we conclude that $\mathcal{B}$ has at least $n+\rho$ components. 
Thus, we can compute the relative complexity of
the log Calabi--Yau pair $(\mathcal{X}',\mathcal{B}'+\mathcal{X}_0')$ over $\{0\} \in \mathbb{D}$. 
We obtain
\[
c_{\{0\}}(\mathcal{X}'/\mathbb{D},\mathcal{B}'+\mathcal{X}_0') =
\dim \mathcal{X}' + \rho(\mathcal{X}'/\mathbb{D}) - |\mathcal{B}'+\mathcal{X}_0'| \leq 
n+1+\rho - (n+\rho+1).
\]
Therefore, by~\cite[Theorem 1]{MS21}, we conclude that 
$(\mathcal{X}',\mathcal{B}'+\mathcal{X}'_0)$ is a formally toric morphism near $\{0\}$. 
Thus, the pair $(\mathcal{X}'_0,\mathcal{B}'_0)$ obtained from adjunction
of $(\mathcal{X}',\mathcal{B}'+\mathcal{X}'_0)$ must be a projective
toric log Calabi--Yau pair. 
Henceforth, we have a finite crepant morpshim 
of log Calabi--Yau pairs 
\[
f_0\colon (\mathcal{X}'_0,\mathcal{B}'_0) 
\rightarrow (\mathcal{X}_0,\mathcal{B}). 
\]
We conclude that the pair $(\mathcal{X}_0,\mathcal{B}_0)$
is a finite quotient of a toric pair. 
\end{proof} 

\begin{proof}[Proof of Theorem~\ref{thm:ct-deg}]
Let $f\colon \mathcal{X}'\rightarrow \mathcal{X}$ be the finite cover given by Lemma~\ref{lem:finite-cover}. 
By Lemma~\ref{lem:finite-cover}.(2), we have an induced finite cover of log Calabi--Yau pairs 
$(\mathcal{X}',\mathcal{B}'+\mathcal{X}'_0)\rightarrow 
(\mathcal{X},\mathcal{B}+\mathcal{X}_0)$. 
Up to shrinking $\mathbb{D}$ near $\{0\}$, we have a commutative diagram 
\[
\xymatrix{
(\mathcal{X},\mathcal{B}+\mathcal{X}_0) \ar[d]_-{\pi} & (\mathcal{X}',\mathcal{B}'+\mathcal{X}'_0) \ar[d]^-{\pi'}\ar[l]_-{f}  \\ 
\mathbb{D} & \mathbb{D}\ar[l]_-{f_\mathbb{D}}
}
\]
satisfying the following conditions:
\begin{enumerate}
\item[(i)]  $\pi'\colon \mathcal{X}' \rightarrow \mathbb{D}$ is a Fano type morphism; 
\item[(ii)] $(\mathcal{X}',\mathcal{B}'+\mathcal{X}'_0)$ is log Calabi--Yau of index one over $\mathbb{D}$; 
\item[(iii)] $(\mathcal{X}',\mathcal{X}_0')$ is plt and purely terminal outside $\mathbb{B}'$;
\item[(iv)] $(\mathcal{X}'_t,\mathcal{B}'_t)$ is of cluster type for $t$ general in $\mathbb{D}$; and
\item[(v)] for $t \in \mathbb{D}$ general, every log canonical center of $(\mathcal{X}'_t,\mathcal{B}'_t)$ is the restriction to 
$\mathcal{X}'_t$ of a log canonical center of $(\mathcal{X'},\mathcal{B}')$.
\end{enumerate}
We may further assume that the following condition is satisfied: 
\begin{enumerate}
\item[(vi)] we have an isomorphism 
\[
{\rm Cl}(\mathcal{X}'/\mathbb{D}) \simeq {\rm Cl}(\mathcal{X}_t) 
\]
induced by restriction
for $t\in \mathbb{D}$ general.
\end{enumerate} 
Note that every log canonical center of $(\mathcal{X},\mathcal{B})$ is horizontal over $\mathbb{D}$.
As $-(K_{\mathcal{X}}+\mathcal{X}_0)$ is ample over $\mathbb{D}$ we may find a boundary $\Delta$ which gives a klt Calabi--Yau pair $(\mathcal{X},\mathcal{X}_0+\Delta)$ over
$\mathbb{D}$. If $\Delta$ is chosen to be general in $|-mK_{\mathcal{X}}|$, then we can assume that the klt Calabi--Yau pair
$(\mathcal{X},\mathcal{X}_0+\Delta)$ is purely terminal on the complement of $\mathcal{B}$. 
We let $(\mathcal{X}',\mathcal{X}_0'+\Delta')$ be the log pull-back of 
$(\mathcal{X},\mathcal{X}_0+\Delta)$ to $\mathcal{X}'$.
Then, the pair $(\mathcal{X}',\mathcal{X}_0'+\Delta')$
is purely terminal on the complement of $\mathcal{B}'$.
In particular, every non-terminal place of 
$(\mathcal{X}',\mathcal{X}_0'+\Delta')$
is horizontal over $\mathbb{D}$
and is a log canonical place of
$(\mathcal{X}',\mathcal{B}'+\mathcal{X}'_0)$. 
Let $\phi\colon \mathcal{Y}\rightarrow \mathcal{X}$ be a $\qq$-factorial 
dlt modification of $(\mathcal{X},\mathcal{B})$
which extracts all non-terminal places of 
$(\mathcal{X}',\Delta'+\mathcal{X}_0')$. 
We write 
\[
\phi^*(K_{\mathcal{X}'}+\mathcal{B}'+\mathcal{X}_0')= 
K_{\mathcal{Y}}+\mathcal{B}_{\mathcal{Y}}+\mathcal{Y}_0
\]
Note that the induced crepant projective birational morphism 
\[
\phi_0 \colon (\mathcal{Y}_0,\mathcal{B}_{\mathcal{Y}_0}) 
\rightarrow (\mathcal{X}_0',\mathcal{B}_0')
\]
only extracts log canonical places of the pair
$(\mathcal{X}_0',\mathcal{B}_0')$.
Analogously, for $t\in \mathbb{D}$ general, the 
crepant projective birational morphism 
\[
\phi_t \colon (\mathcal{Y}_t,\mathcal{B}_{\mathcal{Y}_t}) 
\rightarrow (\mathcal{X}_t',\mathcal{B}_t')
\]
only extracts log canonical places of $(\mathcal{X}_t',\mathcal{B}_t')$.
Therefore, the pair $(\mathcal{Y}_t,\mathcal{B}_{\mathcal{Y}_t})$ is of cluster type by Lemma~\ref{lem:cluster-type-ascends}.
By Lemma~\ref{lem:cluster-type-via-dlt}, we know that there is a dlt modification 
\[
\psi_t\colon (\mathcal{Z}_t,\mathcal{B}_{\mathcal{Z}_t})
\rightarrow 
(\mathcal{Y}_t,\mathcal{B}_{\mathcal{Y}_t})
\]
and a birational contraction 
\[
\chi_t\colon (\mathcal{Z}_t,\mathcal{B}_{\mathcal{Z}_t})
\dashrightarrow 
(T_t,B_{T_t}) 
\]
to a toric log Calabi--Yau pair. 
By condition (v) above, the dlt modification 
$\psi_t$ is indeed the restriction to $\mathcal{Y}_t$
of a dlt modification 
$\psi\colon (\mathcal{Z},\mathcal{B}_{\mathcal{Z}})\rightarrow 
(\mathcal{Y},\mathcal{B}_{\mathcal{Y}})$. 
Let $E_t$ be the exceptional locus of $\chi_t$.
Note that we still have an isomorphism ${\rm Cl}(\mathcal{Z}/\mathbb{D})\simeq {\rm Cl}(\mathcal{Z}_t)$. 
Thus, by~\cite[Lemma 4.2]{JM24}, there exists an effective divisor $E$
on $\mathcal{Z}$ that restricts to $E_t$. 
Further, we may assume that $E$ has as many irreducible components as $E_t$. 
Since $E_t$ has Kodaira dimension zero, we conclude that $E$ has Kodaira dimension zero as well.
On the other hand, we know that 
\[
c(\mathcal{Z}_t,\mathcal{B}_{\mathcal{Z}_t})=|E_t|.
\]
Thus, we conclude that 
\begin{equation}\label{ineq:bound} 
|\mathcal{B}_{\mathcal{Z}}+\mathcal{Z}_0| \geq \dim \mathcal{Z} + \rho(\mathcal{Z}) - |E|.
\end{equation} 
For $\epsilon>0$ small enough, we can define:
\[
\psi^*\phi^*(K_{\mathcal{X}'}+(1-\epsilon)\mathcal{B}'+\epsilon \Delta'+\mathcal{X}'_0) =
K_{\mathcal{Z}}+\Delta_{\mathcal{Z}}+\mathcal{Z}_0, 
\]
so the pair $(\mathcal{Z},\Delta_{\mathcal{Z}}+\mathcal{Z}_0)$ is purely terminal and log Calabi--Yau over $\mathbb{D}$. 
We pick $\delta>0$ small enough and run a 
$(K_{\mathcal{Z}}+\Delta_{\mathcal{Z}}+\mathcal{Z}_0+\delta E)$-MMP over $\mathbb{D}$ which terminates on a model $T/\mathbb{D}$ after contracting all the components of $E$. 
Let $\chi \colon \mathcal{Z}\rightarrow T$ be a birational contraction induced by this MMP. 
Let $(T,B_T+T_0)$ be the pair obtained from pushing forward  $(\mathcal{Z},\mathcal{B}_{\mathcal{Z}}+\mathcal{Z}_0)$ to $T$. 
From inequality~\eqref{ineq:bound}, we conclude that 
\[
|B_T+T_0| \geq \dim T + \rho(T). 
\]
Thus, we have that the pair $(T,B_T+T_0)$ is toric over $\mathbb{D}$.
In particular, the pair $(T_0,B_{T_0})$ obtained from adjunction
to the central fiber is toric. 
We argue that $(\mathcal{Z}_0,\mathcal{B}_{\mathcal{Z}_0})\dashrightarrow 
(T_0,B_{T_0})$ is a birational contraction. 
Indeed, for $\delta>0$ small enough the pair 
$(\mathcal{Z},\Delta_{\mathcal{Z}}+\mathcal{Z}_0+\delta E)$ is purely terminal, so the pair $(\mathcal{Z}_0,\Delta_{\mathcal{Z}_0}+\delta E_0)$ obtained by adjunction to $\mathcal{Z}_0$ is terminal. 
Therefore, the birational map $\chi_0$ cannot extract divisors.
Thus, we conclude that $(\mathcal{X}'_0,\mathcal{B}'_0)$ 
admits a dlt moldification which has a crepant birational contraction to a toric log Calabi--Yau pair $(T_0,B_{T_0})$. 
By Lemma~\ref{lem:cluster-type-via-dlt}, we conclude that 
$(\mathcal{X}_0',\mathcal{B}'_0)$ is a cluster type pair.
Note that we have a crepant finite morphism 
\[
f_0\colon (\mathcal{X}'_0,\mathcal{B}'_0)\rightarrow 
(\mathcal{X}_0,\mathcal{B}_0).
\]
Thus, we conclude that $(\mathcal{X}_0,\mathcal{B}_0)$
is a finite quotient of a cluster type pair.
\end{proof} 

\section{Complements on degenerations of klt surfaces}
\label{sec:complements-degen-klt-surfaces}

In this section, we study degenerations of klt surfaces.
We focus on degenerations of weighted projective planes.
We prove a general statement, Theorem~\ref{thm:1-comp},
that allows us to show that a degeneration of klt surfaces
$\mathcal{X}\rightarrow \mathbb{D}$ admits a $1$-complement.
In order to do so, we need to avoid some special singularities
on the general fiber $\mathcal{X}_t$. 
Our proof for the existence of a $1$-complement 
does not apply in these cases for technical reasons. 
These singularities are introduced in the next definition.

\begin{definition}
\normalfont
We  define the following basket of singularities (up to permutation)\footnote{The notation $[2^k]$ denotes the chain $[2,\cdots,2]$ of $k$ consecutive 2's.}:
\[
\mathcal{F}_1:= \{ [3,2^k] \mid k\in \zz_{\geq 0}\}, 
\]
\[
\mathcal{F}_2:=\{ [4,2^k], [2,3,2^k] \mid k\in \zz_{\geq 0}\}, 
\]
\[ 
\mathcal{F}_3:=\{ [5,2^k], [2,2,3,2^k] \mid k\in \zz_{\geq 0}\}, 
\text{ and }
\]
\[
\mathcal{F}_4:= \{ [6,2^k],[2,2,2,3,2^k],[2,4,2^k],[3,3,2^k] \mid 
k\in \zz_{\geq 0}\}.
\]
The basket of singularities $\mathcal{D}$ 
will denote {\em $D$-type cyclic} singularities, i.e., cyclic quotient singularities with minimal resolution $[2,n,2]$ with $n\geq 2$. These singularities are cyclic quotients; however, they also admit reduced 2-complements as in the case of $D$ singularities.
\label{basket}
\end{definition}

\begin{remark}\label{Wahldeg}
{\em 
The basket of singularities $\mathcal{F}_i$ with $i\in \{1,\dots,4\}$ corresponds precisely to such singularities  that contract torically to a point of orbifold index $i+1$ in such a way
that the contracted curve passes through a unique singular point.
Note that the Wahl singularities appearing in the baskets
$\mathcal{F}_1,\dots,\mathcal{F}_4,\mathcal{D}$ are only finitely many singularities, namely; $[4]$, $[5,2]$, and $[6,2,2]$.
}
\end{remark}

Before establishing the results on the existence of 1-complements on degenerations of surfaces with cyclic quotient singularities, we first state the following preliminary lemmas.

\begin{lemma}
Let $\mathcal{X}$ be $\qq$-Gorenstein variety with at most log canonical singularities and $\mathcal{X}\to\mathbb{D}$ be a proper morphism. Let $0<\varepsilon<1$. If for each $t\neq 0$ we have $mld(\mathcal{X}_t)<\varepsilon$, there exists a divisor $E$ over $\mathcal{X}$ dominating $\mathbb{D}$ for which $a_E(\mathcal{X})< \varepsilon$. 
Furthermore, there exists a projective birational morphism $\mathcal{Y}\rightarrow \mathcal{X}$
from a $\qq$-Gorenstein variey $\mathcal{Y}$ that only extracts $E$.
\label{extraction}
\end{lemma}

\begin{proof}
Let $p\colon \mathcal{Y}\rightarrow \mathcal{X}$ be a log resolution.
For $t$ general, the induced morphism $p_t\colon \mathcal{Y}_t\rightarrow \mathcal{X}_t$ is a log resolution. 
Let $E_1,\dots,E_r$ be the exceptional divisors of $p$ which are horizontal over $\mathbb{D}$
and let $F_1,\dots,F_s$ be the exceptional divisors of $p$ which are vertical over $\mathbb{D}$. 
Write 
\begin{equation}\label{eq:log-discrep}
p^*(K_{\mathcal{X}})=K_{\mathcal{Y}} +\sum_{i=1}^r (1-\alpha_i)E_i + 
\sum_{i=1}^s (1-\beta_i)F_i, 
\end{equation} 
where the $\alpha_i$'s and the $\beta_i$'s are the corresponding log discrepancies.
If we restrict the equality~\eqref{eq:log-discrep} to $\mathcal{X}_t$ general, we obtain 
\[
p_t^*(K_{\mathcal{X}_t})=K_{\mathcal{Y}_t} + \sum_{i=1}^r (1-\alpha_i)E_{i,t}. 
\]
By assumption, we conclude that for each $t$ general, there exists a divisor 
$E_{i,t}$ with $\alpha_i <\epsilon$. 
Thus, we may assume that $\alpha_1<\epsilon$. 
By~\eqref{eq:log-discrep}, we conclude that $a_{\mathcal{X}}(E_1)<\epsilon$ so
such divisor $E$ over $\mathcal{X}$ dominating $\mathbb{D}$ exists. 
The last statement follows from~\cite[Theorem 1]{Mor19}.
\end{proof}

\begin{lemma}
Let $X$ be a normal projective surface with log terminal singularities and $C$ an irreducible rational curve in $X$. If $(X,lC)$ is a klt pair for $l>\frac{5}{6}$, then $(X,C)$ is a log canonical pair. 
\label{lcpairlemma}
\end{lemma}

\begin{proof}
Assume that $(X,lC)$ is klt near $x$ a closed point. 
We argue that $(X,C)$ is lc near $x$. 
Let $t\in [0,1]$ be a parameter.
Let $\phi\colon Y\rightarrow X$ be a finite quasi-\'etale cover near $x$ 
such that $Y$ is smooth over the  $y$ preimage of $x$. 
Let $\phi^*(K_X+tC)=K_Y+tC_Y$. 
The log canonical condition for pairs behaves well under quasi-\'etale finite covers
so $(X,tC)$ is lc if and only if $(Y,tC_Y)$ is lc.
Thus, it suffices to show the statement for smooth germs.  
In this case, it is known that the reduced curves with the
largest log canonical thresholds are smooth and nodal curves, with threshold one,
and the cuspidal singular curve with threshold $\frac{5}{6}$.
Therefore, if $(X,lC)$ is lc for $l>5/6$, then 
$(Y,lC_Y)$ is lc for $l>5/6$ which implies that $C_Y$ is either smooth or nodal at $y$.
Hence, $(Y,C_Y)$ and $(X,C)$ are both lc.
\end{proof}

\begin{lemma}\label{lem:min-resol-horizontal}
Let $\mathcal{X}\rightarrow \mathbb{D}$ be a $\qq$-Gorenstein klt degeneration of a rational normal projective surface, with $\mathcal{X}_0$ reduced and $\rho(\mathcal{X}_t)=\rho(\mathcal{X}_0)$.
Let $\pi\colon \mathcal{Y}\rightarrow \mathcal{X}$ be a projective birational morphism such that all whose exceptional divisors $E$ are $\qq$-Cartier, horizontal over $\mathcal{X}$ and satisfy $a_E(\mathcal{X})\leq 1$. 
Assume that $\mathcal{Y}_t$ is smooth for $t\neq 0$.
Then, for each exceptional prime divisor $E$ of $\pi$, we have that $E_0:=E\cap \mathcal{Y}_0$ is an irreducible curve. In addition, if $E_t^2<-1$ for $t\neq 0$, then $E_0$ contains most at most two Wahl singular points $P_0$ and $P_1$, and the pair $(\mathcal{Y}_0,E_0)$ is plt near $P_i$ $(i=0,1)$. 
\end{lemma} 

\begin{proof}
Note that is a $\qq$-Gorenstein smoothing $\mathcal{Y}\to\mathbb{D}$. Since $\mathcal{Y}_0$ contains at most T-singularities and $\rho(\mathcal{Y}_t)=\rho(\mathcal{Y}_0)$, Noether's formula \ref{Noether} implies that $\mathcal{Y}_0$ has at most Wahl singularities. Moreover, as $H^2(\mathcal{Y}_0,\oo_{\mathcal{Y}_0})=H^1(\mathcal{Y}_0,\oo_{\mathcal{Y}_0})=0$, the short exact sequence \ref{shortexact} yields an isomorphism 
$\Pic(\mathcal{Y}_t)\cong \Pic(\mathcal{Y}_0)$
for $t\neq 0$, induced by the specialization map into $\Cl(\mathcal{Y}_0)$. Let $E$ be a prime exceptional divisor, it follows that the generic point $t$ of $\mathbb{D}$, the fiber $E_t=E\cap \mathcal{Y}_t$ is irreducible. Moreover, the condition $a_E(\mathcal{X})\leq 1$ implies that the pair $(\mathcal{Y}_t,E_t)$ is log canonical. Hence, by the classification of two-dimensional log canonical pairs with reduced boundary \cite[Proposition 3.2.7]{Ale92}, we deduce that $E_t$ is a smooth rational curve. If $E_0$ is reducible, say $E_0=D_1+D_2$, then $D_1\cdot D_2>0$ and thus $D_1,D_2$ are linearly independent in $\Pic(\mathcal{Y}_0)\otimes_{\zz}\qq$. This contradicts the equality $\rho(\mathcal{Y}_t/\mathcal{X}_t)=\rho(\mathcal{Y}_0/\mathcal{X}_0)=\rho(\mathcal{Y}/\mathcal{X})$ and consequently $E_0$ must be irreducible. \\
Furthermore, by \cite[Lemma 2]{Man91} the restriction map $\Pic(\mathcal{Y})\to \Pic(\mathcal{Y}_0)$ is an isomorphism. Hence, the intersection product is preserved under the specialization map $\Pic(\mathcal{Y}_t)\to\Cl(\mathcal{Y}_0)$. Indeed, given Cartier divisor $L$ in $\mathcal{Y}_0$, there exists $\mathcal{L}\in \Pic(\mathcal{Y})$ such that $\mathcal{L}|_{\mathcal{Y}_0}=L$. Then for $m\gg 0$, $m(\mathcal{L}|_{\mathcal{Y}_t}\cdot E_t)=m(\mathcal{L}\cdot E\cdot \mathcal{Y}_t)=m(\mathcal{L}\cdot E\cdot \mathcal{Y}_0)=m(L\cdot kE_0)$, where $E|_{\mathcal{Y}_0}=kE_0$ for some $k\geq 1$. If $E_t^2<-2$, then for $\pi$ restricted to a neighborhood $Z_0$ of $E_0$, it follows that $K_{Z_0}$ is $\pi|_{Z_0}$-ample. By \cite[Lemma 3.14]{KSB88}, $E_0$ contains at most two singular points and, near each $P_i$, the curve $E_0$ forms an analytic coordinate axis. In the case $E_t^2=-2$, then $E_0\cdot K_{\mathcal{Y}_0}=0$. In this situation, $E_0$ is either smooth or contains exactly two Wahl singularities of the same type $P$. Again, $E_0$ is forced to form a toric coordinate axis near $P_i$.
\end{proof}

\begin{lemma}\label{lem:sing-central-fiber}
Let $\mathcal{X}\rightarrow \mathbb{D}$ be a $\qq$-Gorenstein degeneration of a rational normal projective surface. Assume that $\mathcal{X}_0$ has log terminal singularities and $\rho(\mathcal{X}_t)=\rho(\mathcal{X}_0)$.
Then, the limit in $\mathcal{X}_0$ of any quotient singularity of $\mathcal{X}_t$ is a quotient singularity whose dual graph has the same Dynkin type ($A$, $D$ or $E$). Furthermore, the singularity basket of $\mathcal{X}_0$ consists of the singularity limits of $\mathcal{X}_t$ and possibly some Wahl singularities.
\end{lemma} 
\begin{proof}
By Lemma \ref{extraction}, there exists a projective birational morphism $\mathcal{Y}\to\mathcal{X}$ such that, for every $t\neq 0$ $\mathcal{Y}_t\to \mathcal{X}_t$ is the minimal resolution of $\mathcal{X}_t$. By construction, the exceptional divisors $E$ of this morphism are horizontal over $\mathcal{X}$ and satisfy $a_E(\mathcal{X})<1$. Hence, by Lemma \ref{lem:min-resol-horizontal}, each curve $E_0=E_0|_{\mathcal{Y}_0}$ contains at most two singular points $P_0,P_1$. Since $\Pic(\mathcal{Y}_t)\to \Cl(\mathcal{Y}_0)$ preserve products, it follows that $E_t\cdot E_t^\prime=E_0\cdot E_0^\prime$ for any pair $E,E^\prime$ of exceptional divisors. This implies that  over any singular point $P$ of $\mathcal{X}_t$, if $E_t\cdot E^\prime_t$=1, then $(E_0+E_0^\prime)_{red}$ forms a toric boundary near some Wahl singularity $P_0=\frac{1}{n^2}(1,na-1)$. In particular, it follows that the multiplicities of $E_0$ and $E_0^\prime$ are equal to $n$. This implies that for any singular point $P\in\mathcal{X}_t$, the exceptional divisors in $\mathcal{Y}_0$ preserve the Dynkin type of its dual graph, implying that $P$ degenerates to a singular point $P_0$ with the same Dynkin type. By the same argument of Lemma \ref{lem:min-resol-horizontal}, $\mathcal{X}_0$ may also acquire some additional Wahl singularities. 
\end{proof}

\begin{lemma}\label{lem:rigid-sing}
The singularities in the baskets $\mathcal{F}_i$ with $i\in \{1,\dots,4\}$ are $\qq$-Gorenstein rigid, except for the cases $[4],[5,2],[3,3],[6,2,2]$ and $[2,4,2]$. In the setting of Lemma \ref{lem:sing-central-fiber}, if $\mathcal{X}_0$ contains a singularity of type $[2,4,2]$, then $\mathcal{X}_t$ must contain a point of the same isomorphism type for $t\neq 0$.
\end{lemma}
\begin{proof}
For integers $0<q<m$ with $\operatorname{gcd}(m,q)=1$, set $k=\operatorname{gcd}(m,q+1)$, so that we can write $m=kr$ and $q+1=kc$. Then, the cyclic quotient singularity $\frac{1}{m}(1,q)$ can be expressed as $\frac{1}{kr}(1,kc-1)$. By \cite{ACC16}, this singularity is $\qq$-Gorenstein rigid if and only if $k<r$. We apply this criterion to the singularities of the form $[2,4,2^l]$ with $l\geq 0$, leaving the remaining cases to the reader. The continued fraction $[2,4,2^l]$ corresponds to $\frac{7+5l}{4+3l}$. Hence, we compute,
\[
k=\operatorname{gcd}(7+5l,5+3l) = 
\begin{cases} 
4 & \text{if } l \equiv 1 \pmod{4},\\
2 & \text{if } l \equiv 3 \pmod{4},\\
1 & \text{if } l \equiv 0,2 \pmod{4}.
\end{cases}
\]
We observe that the only case with $k\geq r$ is precisely when $l=1$. \newline
For the singularity $P$ of type $[2,4,2]$ we compute all possible partial resolutions $f:(C\subset Z)\to (\mathcal{X}_0,P)$ such that $Z$ contains at most Wahl singularities and $K_Z$ is $f$-nef. According to \cite[Section 1.3.2]{BC94}, there are just two such resolutions, corresponding to the zero continued fractions $[1,2,1]$ and $[2,1,2]$ (see for example \cite[Section 2]{UZ24}). Therefore, $f$ is either the minimal resolution of $P$ or $f$ contracts a reducible curve $C$ consisting of two components intersecting at a $\frac{1}{4}(1,1)$ singularity. The latter possibility is excluded in the setting of \ref{lem:min-resol-horizontal}, since we consider deformations such that $\rho(\mathcal{X}_t)$ is constant. Hence, the claim follows.
\end{proof}

In the following theorem, we make use of the baskets in Definition~\ref{basket}.

\begin{theorem}\label{thm:1-comp}
Let $\pi\colon \mathcal{X}\rightarrow \mathbb{D}$ be a Fano projective morphism where $\mathcal{X}_0$ is reduced in $\mathcal{X}$ and
$(\mathcal{X},\mathcal{X}_0)$ plt.
Assume that the following conditions hold for the general fiber $\mathcal{X}_t$ with $t\neq 0$:
\begin{enumerate}
\item $\mathcal{X}_t$ is a toric surface of Picard rank one;
\item ${\rm mld}(\mathcal{X}_t)<\frac{1}{6}$; 
\item $\mathcal{X}_t$ does not have singularities in the baskets 
$\mathcal{F}_1,\dots,\mathcal{F}_4,\mathcal{D}$.
\end{enumerate}
Then, the pair $(\mathcal{X},\mathcal{X}_0)$ admits a $1$-complement over $\mathbb{D}$.
\end{theorem}
\begin{proof}
We break the proof into five steps.\\

\textit{Step 1}: In this step, we construct a partial resolution $\mathcal{Y}\to\mathcal{X}$ of the family $\mathcal{X}\rightarrow \mathbb{D}$ by extracting the curve computing the minimal log discrepancy along the general fibers.\\

Let $Q$ be a singular point of the general fiber $\mathcal{X}_t$
which computes ${\rm mld}(\mathcal{X}_t)$. Consider the extraction of a curve $E_t$ chosen from the minimal resolution of $Q$ attaining the lowest mld. We denote by $\mathcal{Y}_t\to\mathcal{X}_t$ the corresponding projective birational morphism. By Lemma~\ref{extraction}, there exists a plt $\mathbb{Q}$-Gorenstein family $\mathcal{Y}\to \mathbb{D}$ whose general fiber is $\mathcal{Y}_t$ together with a divisorial contraction $(\mathcal{E}\subset \mathcal{Y})\to \mathcal{X}$ over $\mathbb{D}$ such that $\mathcal{E}|_{\mathcal{Y}_t}=E_t$ for $t\neq 0$. 
By Lemma~\ref{lem:min-resol-horizontal}, we know that $E_0=\mathcal{Y}_0\cap \mathcal{E}$ is a smooth rational curve. Indeed, $E_0$ is the normal image
of a smooth rational curve on a horizontal resolution
of $\mathcal{X}\rightarrow \mathbb{D}$.
By Lemma~\ref{lem:sing-central-fiber}, the central fiber $\mathcal{X}_0$ has cyclic quotient singularities and so $E_0$ contracts to a cyclic quotient point, it follows that $E_0$ contains at most two singular points $Q_0,Q_1$ of $\mathcal{Y}_0$ for which the apir $(\mathcal{Y}_0,E_0)$ is plt near $Q_i$ for $(i=0,1)$.\\

\textit{Step 2}: In this step, we study the anti-nef threshold of the pair $(\mathcal{Y}_0,E_0)$ and the induced morphism on the central fiber.\\

We now consider the klt pair $(\mathcal{Y}_0,dE_0)$, where $d$ denotes an anti-nef threshold, i.e., the supremum of $d>0$ for which 
$(\mathcal{Y}_0,dE_0)$ has klt singularities and 
$-(K_{\mathcal{Y}_0}+dE_0)$ is nef. If $d=1$ the theorem claim follows directly from \textit{Step 5} below applied to $Y_0$, since $-(K_{\mathcal{Y}_0}+E_0)$ is nef. Thus, we may assume that $d<1$. By construction, there exists an extremal curve $C$ such that $(K_{\mathcal{Y}_0}+dE_0)\cdot C=0$. We aim to show that the contraction induced by $C$ is a birational morphism. Suppose instead that $C$ is contracted by a $\pp^1$-fibration $\mathcal{Y}_0\to \mathcal{Y}_0^\prime$ where $\mathcal{Y}_0^\prime$ is a curve.
Since $-(K_{\mathcal{Y}_t}+dE_t)$ is nef and $\mathcal{Y}_t$ is toric for $t\neq 0$, it follows that $-(K_{\mathcal{Y}_t}+dE_t)$ is semiample on the general fiber. For $m\gg 0$ and $t\neq 0$, the morphism induced by $|-m(K_{\mathcal{Y}_t}+dE_t)|$ must be birational. Otherwise, it would induce a toric fibration $\mathcal{Y}_t\to \mathcal{Y}^\prime_t$ for which $E_0\cdot F=1$ for a fiber $F$. This would give $F\cdot (K_{\mathcal{Y}_t}+dE_t)=2-d=0$ which contradicts $d<1$. The linear system $|-m(K_{\mathcal{Y}}+d\mathcal{E})|$ is basepoint free, hence it induces a morpshim $\mathcal{Y}\to \mathcal{Y}^\prime=\pp_{\mathbb{D}}(\oo(-m(K_{\mathcal{Y}}+d\mathcal{E})))$ which further induces a flat morphism $\mathcal{Y}^\prime\to \mathbb{D}$. This is a contradiction of upper semicontinuity of fiber dimensions, since ${\rm dim } \mathcal{Y}_0^\prime < {\rm dim  } \mathcal{Y}_t$ for $t\neq 0$. Hence, the contraction $\varphi:\mathcal{Y}_0\to \mathcal{Y}_0 ^\prime$ of $C$ is a birational morphism, for which we define $P=\varphi(C)$.\\

\textit{Step 3}: In this step, we study the geometry of the contraction of the curve $C$ and the singularity of $E_0'$ at the image of $C$.\\ 

By Lemma \ref{lcpairlemma}, since the pair $(\mathcal{Y}_0^\prime, lE_0^\prime)$ is klt for $l>\frac{1}{6}$, it follows that $(\mathcal{Y}_0^\prime, E_0^\prime)$ is log canonical. First, we study the existence of complements for the pair $(\mathcal{Y}_0^\prime, E_0^\prime)$ around $P$ when the pair is lc but not klt at $P$. In this situation, $P$ is a log canonical center of $(\mathcal{Y}_0^\prime, E_0^\prime)$. Then by the classification two-dimensional of log canonical pairs with reduced boundary, we fall into the Cases $9$ and $10$ of Fig 3. in \cite[Proposition 3.2.7]{Ale92} (see also \cite[Theorem 4.15]{KM92}). We study both of them separately:\\ 

\textit{Case 3.1}: We assume that $\mathcal{Y}_0^\prime$
has a $D$ singularity at $P$, i.e., its resolution graph is a D Dynkin diagram.\\

The curve $E_0^\prime$ has exactly one local component through $P$. In this case we have that the singular point $P$ is a D singularity. By \cite[Theorem 4.15]{KM92} in its minimal resolution, the strict transform of $E_0$ intersects transversely the initial curve $\Gamma_0$ in a chain $\Gamma_0-\dots-\Gamma_r$ of rational curves. We denote the $(-2)$ curves forming the fork as $F_0,F_1$, both intersecting the curve $\Gamma_r$. The point $P$ admits a local 2-complement, indeed $2(K_{\mathcal{Y}_0^\prime}+E_0^\prime)\sim 0$. Moreover, since $\varphi^*(K_{\mathcal{Y}_0^\prime}+E_0^\prime)\sim K_{\mathcal{Y}_0}+E_0+aC$, it follows that either $a=\frac{1}{2}$ or $a=1$. The case $a=1$ cannot occur. Indeed, in the minimal resolution of $\mathcal{Y}_0$, the strict transform $\hat{C}$ of $C$ would have to lie in a chain of rational curves obtained by successively blowing up infinitely close points over a single point $x$ belonging to the set
$$\mathcal{B}=\{E_0\cap \Gamma_0,\Gamma_0\cap\Gamma_1,\dots,\Gamma_{r-1}\cap\Gamma_r\}.$$
In this situation, the curve is forced to contain a $D$ singularity, which contradicts condition (3). Also, since $C$ is extremal and $K_{Y_0}\cdot C<0$, the curve $\hat{C}$ satisfies that $\hat{C}^2=-1$. This implies that $C$ cannot appear in the graph of the minimal resolution of the point $P$.\newline
The case $a=\frac{1}{2}$ cannot occur. Indeed, in the minimal resolution of $\mathcal{Y}_0$, the curve $C$ would be contained in a chain of rational curves lying over $\Gamma_r\cap F_0$ or $\Gamma_r\cap F_1$. In this situation, $C$ must contain either a singularity of the form $[3,2,\dots,2]$, belonging to $\mathcal{F}_3$ in the basket of singularities of (3), or a $D$ singularity. In the first case, singularities of type $[3,2,\dots,2]$ are not Wahl, by Lemma \ref{lem:rigid-sing} these singularities cannot appear in $\mathcal{Y}_0$. In the second case a $D$ singularity is also excluded. Therefore, Case 3.1 cannot occur.\\

\textit{Case 3.2:} We assume that $\mathcal{Y}_0^\prime$
has a cyclic quotient singularity at $P$ and $E_0^\prime$
has two local components around $P$.\\
 
The curve $E_0^\prime$ has two local components through $P$. This implies that $E_0$ intersects $C$ in two points and hence $E_0^\prime$ is a rational nodal curve. By the projection formula, we have that $(K_{\mathcal{Y}_0^\prime}+E_0^\prime)\cdot E_0^\prime=(K_{\mathcal{Y}_0}+E_0)\cdot E_0\leq 0$. On the other hand, the adjunction formula gives $(K_{\mathcal{Y}_0^\prime}+E_0^\prime)\cdot E_0^\prime=\operatorname{Diff(0)}_{E_0^\prime}\geq 0$, implying that $(K_{\mathcal{Y}_0^\prime}+E_0^\prime)\cdot E_0^\prime=0$. Since $\rho(\mathcal{Y}_0^\prime)=1$, it follows that $(\mathcal{Y}_0^\prime,E_0^\prime)$ is an index 1 log Calabi-Yau pair.\\

\textit{Case 3.3:} We assume that $(\mathcal{Y}_0^\prime,E_0^\prime)$
is plt near $P$.\\

Now, let $(Y_0^\prime,E_0^\prime)$ be a plt pair near $P$. We proceed to show that $\varphi$ is a formally toric morphism near $P$. By \cite[Proposition 3.2.7]{Ale92} it follows that $(\mathcal{Y}_0^\prime,E_0^\prime)$ is toric near $P$. In particular, there exists a curve $F_0^\prime\subset \mathcal{Y}_0^\prime$ such that $(\mathcal{Y}_0^\prime,E_0^\prime+F_0^\prime)$ is log canonical and $P$ is a log canonical center of this pair. Moreover, in a neighborhood of $P$ the pair $(\mathcal{Y}_0^\prime,E_0^\prime+F_0^\prime)$ is log Calabi-Yau of index 1, so for any exceptional divisor $E$ over $P$ in a log resolution of $(\mathcal{Y}_0^\prime,P)$ the log discrepancy satisfies $a(E,\mathcal{Y}_0^\prime,E_0^\prime+F_0^\prime)\in\mathbb{Z}_{\geq 0}$. By construction $\operatorname{coeff}_{C}(K_{\mathcal{Y}_0}+dE_0-\varphi^*(K_{\mathcal{Y}_0^\prime}+E_0^\prime))=0$, so that $a(C,\mathcal{Y}_0^\prime,dE_0^\prime)=1$. Since $C\cdot((1-d)E_0+F_0)>0$, we deduce that $a(C,\mathcal{Y}_0^\prime,E_0^\prime+F_0^\prime)<1$ and consequently, $a(C,\mathcal{Y}_0^\prime,E_0^\prime+F_0^\prime)=0$. This implies that the pair $(\mathcal{Y}_0,E_0+F_0+C)$ is log Calabi-Yau over $P$. Following~\cite[Definition 3.15]{MS21}, we compute the relative complexity of $(\mathcal{Y}_0,E_0+F_0+C)$ over $P$: 
\begin{equation}
c_P(\mathcal{Y}_0/\mathcal{Y}_0^\prime, E_0+F_0+C)=\operatorname{dim}(\mathcal{Y}_0)+\rho(\mathcal{Y}_0/\mathcal{Y}_0^\prime)-|E_0+F_0+C|.
\end{equation}
In this case, we observe that $c_P(\mathcal{Y}_0/\mathcal{Y}_0^\prime)=0$. Therefore, by \cite[Theorem 1]{MS21}, it follows that $\varphi$ is formally toric at $P$, and consequently the divisor $E_0+F_0+C$ forms a formally toric boundary over $P$. This implies that the only possible singular points of $C$ are $P_0=E_0\cap C$ and $P_1=F_0\cap F$.
We conclude that the pair $(\mathcal{Y}'_0,E'_0)$ is plt.\\

\textit{Step 4:} In this step, we conclude that $\mathcal{Y}^\prime_0$ has at most two singular points along $E^\prime_0$. Further, we argue that either
the pair $(\mathcal{Y}^\prime_0,E^\prime_0)$ is plt or $E^\prime_0$ is a nodal rational curve.\\ 

Recall, we have a projective birational contraction
$\mathcal{Y}_0\rightarrow \mathcal{Y}_0'$
that contracts the curve $C$ to $P$. 
In \textit{Step 3}, we showed that \textit{Case 3.1} above does not occur,
while in \textit{Case 3.2} the pair $(\mathcal{Y}_0',E_0')$ is 
a log Calabi--Yau pair of index one and $E_0'$ is a nodal rational curve.
Thus, it suffices to consider \textit{Case 3.3}, where the pair $(\mathcal{Y}_0^\prime,E_0^\prime)$ is plt. The contraction
$\mathcal{Y}_0\rightarrow \mathcal{Y}_0'$ is formally toric near $P$.
Thus, if $E_0\cap C$ is a singular point, then it follows
that $E'_0$ contains at most two singular points of $\mathcal{Y}'_0$. The step is finished in this case.
From now on, we assume that $E_0\cap C$ is a smooth point. 
We argue that, under our assumptions, the orbifold index of $\mathcal{Y}'_0$ 
at $P$ is at least $6$.

Assume instead that $i_P(\mathcal{Y}'_0)\leq 5$.
Then either $\mathcal{Y}_0'$ is smooth along $C$ or contains at most $A_n$ singularities; 
or else $C$ passes through a singular point $P\cap F_0$
belonging to $\mathcal{F}_i$ with $i\in \{1,\dots,4\}$.
In the case that $\mathcal{Y}_0'$ is smooth along $C$ or passes through an $A_n$ singularity, since $C$ is a $K_{\mathcal{Y}_0}$-negative extremal curve, it follows that $K_{\mathcal{Y}_0}\cdot C=-1$. As $(K_{\mathcal{Y}_0}+dE_0)\cdot C=0$, we obtain that $d=1$ and the theorem follows from \textit{Step 5} below. Now, suppose that $C$ passes through a non $A_n$ singular point of $\mathcal{Y}_0$.
Since $\mathcal{X}_t$ does not have singularities in 
$\mathcal{F}_i$ with $i\in \{1,\dots,4\}$, by Lemma \ref{lem:rigid-sing},
we note that $C$ passes through a Wahl singularity
in the baskets $\mathcal{F}_i$ with $i\in \{1,\dots,4\}$, see Remark \ref{Wahldeg}, or possibly through a T-singularity $[3,3]$. 
Hence, $C$ can only pass through singularities in 
\[
\mathcal{S}:= 
\{ [4],[3,3],[5,2],[6,2,2]\}. 
\]
For each singularity in the basket $\mathcal{S}$, we compute the corresponding nef threshold:
\begin{itemize}
\item If $C$ passes through a singularity $[4]$, then $d=\frac{1}{2}$.
\item If $C$ passes through a singularity $[3,3]$, then $d=\frac{1}{2}$.
\item If $C$ passes through a singularity $[5,2]$, then $d=\frac{2}{3}$.
\item If $C$ passes through a singularity $[6,2,2]$, then $d=\frac{3}{4}$.
\end{itemize}
Let $\lambda={\rm mld}(\mathcal{X}_t)$. By Lemma \ref{lem:sing-central-fiber}, the divisor $-(K_{\mathcal{Y}_0}+(1-\lambda)E_0)$ is nef, which implies that $1-\lambda\leq d$. Since $\lambda<\frac{1}{6}$, this leads to a contradiction. Therefore, $C$ must contract to a point of orbifold index greater than or equal $6$.\\

\textit{Step 5}: We construct a 1-complement of the pair $(\mathcal{Y}_0^\prime,E_0^\prime)$, we argue the existence of a 1-complement of $(\mathcal{X},\mathcal{X}_0)$.\\

Suppose the pair $(\mathcal{Y}_0^\prime,E_0^\prime)$ is plt. First, we analyze the existence of the 1-complement for the case where $Q_0=E_0\cap C$ is a singular point. If the points $P$ and $Q_1$ have orbifold indices $m,m_1$ respectively, we note that the different of $E_0^\prime$ is given by:
\begin{equation}
\operatorname{Diff(0)}_{E_0^\prime}=\left(1-\frac{1}{m}\right)P+\left(1-\frac{1}{m_1}\right)Q_1.
\end{equation}
For the pair $(E_0^\prime,\operatorname{Diff(0)}_{E_0^\prime})$, the 1-complement is given by the divisor $\{P\}+\{Q_1\}$. Since $\rho(Y_0^\prime)=1$, then $-(K_{\mathcal{Y}_0^\prime}+E_0^\prime)$ is
either ample or numerically trivial. As $E_0^\prime\cdot (K_{\mathcal{Y}_0^\prime}+E_0^\prime)<0$, it must be ample. By the Kawamata-Viehweg vanishing theorem, the restriction map $H^0(\mathcal{Y}_0^\prime,-(K_{\mathcal{Y}_0^\prime}+E_0^\prime))\to H^0(E_0^\prime,-(K_{E_0^\prime}+E_0^\prime)|_{E_0^\prime})$ is surjective, so the 1-complement lifts to a divisor $D\in |-(K_{\mathcal{Y}_0^\prime}+E_0^\prime)|$. By inversion of adjunction \cite[Theorem 5.50]{KM92}, it follows that the pair $(\mathcal{Y}_0^\prime, D+E_0^\prime)$ is log canonical, hence it is an index 1 log Calabi-Yau pair. Moreover, we have $a(C,\mathcal{Y}_0^\prime,D+E_0^\prime)=1$ and consequently that $\varphi^*(D)+E_0+C\sim -K_{\mathcal{Y}_0}$. By pushing forward $\varphi^*(D)+E_0+C$ under the map $\mathcal{Y}_0\to \mathcal{X}_0$, we obtain an 1-complement for $\mathcal{X}_0$. By construction, each fiber $\mathcal{X}_t$ is a principal divisor on $\mathcal{X}$ for any $t$. By \cite[Lemma 2.1]{EV85} we have that $K_{\mathcal{X}}|_{\mathcal{X}_t}\sim K_{\mathcal{X}_t}$. In particular, this implies that $\operatorname{Diff}_{\mathcal{X}_0}(0)\sim 0$. Following \cite [Proposition 3.7]{PS09}, we obtain the existence of the 1-complement of the pair $(\mathcal{X},\mathcal{X}_0)$ over $\mathbb{D}$.\\ 
Now suppose that $E_0\cap C$ is a smooth point. If some $Q_i$ is smooth, the existence of the 1-complement of $(\mathcal{Y}_0^\prime,E_0^\prime)$ is again obtained by the previous case. If both $Q_i$ are singular, since $\operatorname{deg}(K_{E_0^\prime}+\operatorname{Diff(0)}_{E_0^\prime})<0$, the curve $E_0$ is forced to contain two $A_1$ points, which implies that $\mathcal{X}_t$ contains a singularity in the basket $\mathcal{D}$. This again contradicts condition (3).
\end{proof}

Interestingly, under the same assumptions as in Theorem \ref{thm:1-comp}, we can further show that every singularity appearing in a general fiber $\mathcal{X}_t$ specializes to a singularity of the same isomorphism type in the central fiber $\mathcal{X}_0$.
\begin{corollary}\label{cor:const-sing}
Let $\pi\colon \mathcal{X}\rightarrow \mathbb{D}$ be a Fano projective morphism where $\mathcal{X}_0$ is reduced in $\mathcal{X}$ and
$(\mathcal{X},\mathcal{X}_0)$ plt. Assume that $\mathcal{X}_t$ satisfies conditions (1) to (3) of Theorem \ref{thm:1-comp} and $\mathcal{X}_t$ contains no $A_n$ singularities. Then, the singularity basket of $\mathcal{X}_0$ agrees with that of $\mathcal{X}_t$, up to possibly adding some Wahl singularities.
\end{corollary}

\begin{proof}
By Theorem~\ref{thm:1-comp}, we know that the family admits a $1$-complement 
$(\mathcal{X},\mathcal{B})$ over $\mathbb{D}$.
Let $C$ be a singular curve of $\mathcal{X}$ and $C_0$ be the intersection of $C$ with $\mathcal{X}_0$. 
We argue that, up to finitely many points of the base, the fibers $\mathcal{X}_t$ have the same singularity as $\mathcal{X}_0$. 
As $C_0 \in \mathcal{X}_0$ is a singular point which is not a canonical Gorenstein singularity, we conclude that $\mathcal{B}_0$ must pass through $C_0$. 
Furthermore, $C_0$ is not a $D$ type singularity and so $\mathcal{B}_0$ must have two branches $\mathcal{B}_{0,1}$ and $\mathcal{B}_{0,2}$ near $C_0$. 
Note that both $\mathcal{B}_{0,1}$ and $\mathcal{B}_{0,2}$ are analytically Cartier divisors near $C_0$. Let $\mathcal{B}_t$ be the restriction of $\mathcal{B}$ to $\mathcal{X}_t$ arbitrary. Again, we know that $C_t$ is neither a canonical Gorenstein singularity 
nor a $D$ type singularity, hence $\mathcal{B}_t$ has two analytic Cartier branches $\mathcal{B}_{t,1}$ and $\mathcal{B}_{t,2}$. 
We conclude that $\mathcal{B}$ has two analytically $\qq$-Cartier branches 
$\mathcal{B}_1$ and $\mathcal{B}_2$ near $C_0\in \mathcal{X}_0$ (see, e.g.,~\cite[Lemma 31.18.9]{Stack}).
Thus, locally analytically near $C_0$ the pair $(\mathcal{X},\mathcal{B}+\mathcal{X}_0)$ has three $\qq$-Cartier components, namely $\mathcal{B}_1,\mathcal{B}_2$ and $\mathcal{X}_0$.
Hence, the point $C_0$ is a toric threefold singularity of $\mathcal{X}_0$,
indeed, is a point of complexity zero (see, e.g.,~\cite[Theorem 1.1]{MS21}).
As $\mathcal{X}_0$ is Cartier, we conclude that $\mathcal{X}_0$ is isomorphic
to a product $\mathcal{X}_0\times \mathbb{D}$ locally analytically near $C_0$. 

Now, let $C_0\in \mathcal{X}$ be a closed point at which $\mathcal{X}_0$ is singular. By assumption, no singular curve $C$ of $\mathcal{X}$ is contained in $\mathcal{X}_0$. Thus, either there exists a singular curve $C$ on $\mathcal{X}$ that dominates $\mathbb{D}$ and its intersection with $\mathcal{X}_0$ is $C_0$
or $\mathcal{X}$ is smooth in a punctured disk near $C_0$. 
In the first case, by the previous paragraph, we conclude the singularity at 
$C_0\in \mathcal{X}_0$ is isomorphic to a singularity of $\mathcal{X}_t$.
In the second case, the singularity $C_0\in \mathcal{X}_0$ admits a $\mathbb{Q}$-Gorenstein smoothing, and so it is a Wahl singularity.
This finishes the proof of the corollary.
\end{proof}

\begin{corollary}
Let $\pi\colon \mathcal{X}\rightarrow \mathbb{D}$ be a Fano projective morphism with $(\mathcal{X},\mathcal{X}_0)$ plt and $\mathcal{X}_0$ reduced. Assume that the general fiber $\mathcal{X}_t$ is isomorphic to a weighted projective plane $\pp(a,b,c)$, where $(a,b,c)$ is well formed. Then, the pair $(\mathcal{X},\mathcal{X}_0)$ admits a 1-complement, except possibly in the situation $\operatorname{mld}(\mathcal{X}_t)\geq \frac{1}{6}$ for $t\neq 0$, and in the following cases (up to a permutation of the weights): Let $n\geq 2$,
\begin{enumerate}
\item $\mathcal{X}_t=\mathbb{P}(1+4l(n-1),2n-1+4k(n-1),4n-4)$, where $0\leq l,k<n-1$.
\item $\mathcal{X}_t=\mathbb{P}(1+l(6n-5),3n-1+k(6n-5),6n-5)$, where $0\leq l,k<\lceil\frac{4(6n-5)}{9}\rceil$.
\item $\mathcal{X}_t=\mathbb{P}(1+l(6n-7),3n-2+k(6n-7),6n-7)$, where $0\leq l,k<\lceil\frac{4(6n-7)}{9}\rceil$.
\end{enumerate}
\label{weightedlist}
\end{corollary}

\begin{proof}
The case $\rm{mld}(\mathcal{X}_t)\geq \frac{1}{6}$ correspond to finitely many cases by \cite[Proposition 5]{BB92}. Let us assume $\rm{mld}(\mathcal{X}_t)<\frac{1}{6}$ for $t\neq 0$. In the setting of the previous theorem, the pair $(E_0^\prime,\operatorname{Diff}(0)_{E_0^\prime})$ fails to admit a 1-complement if and only if $E_0^\prime$ contains three singular points with orbifold indices $(m,n_0,n_1)$ belonging to one of the following  triples: $(m,2,2)$ for $m\geq 2$, $(5,2,3)$, $(4,2,3)$, $(3,2,2)$. Here, $m$ denotes the index of the singularity obtained by contracting $C$. Consequently, by Lemma \ref{lem:rigid-sing} the singularity $Q$ that realizes $\operatorname{mld}(\mathcal{X}_t)$, must be either of type $D$, corresponding to a triple $(m,2,2)$, or else of the form $[2,n,3]$, $[2,n,2,2]$ with $n\geq 2$ for the other cases. We prove that cases $(1)$, $(2)$, $(3)$ correspond respectively to the only possible weighted projective planes that satisfy this property. 
We now establish the bound for $k$ in (1); The other cases follow analogously. Following \cite[Section 3]{Am06}, in a weighted projective plane $\pp(a_0,a_1,a_2)$ with $(a_0,a_1,a_2)$ well-formed, the minimal log discrepancy of a singular point $P_i=\frac{1}{a_i}(a_j,a_k)$ is given by the formula $$\operatorname{mld}(P_i)= \min_{1\leq t\leq a_i-1}\left\{\tfrac{t a_j}{a_i}\right\}+ \left\{\tfrac{t a_k}{a_i}\right\},$$ where $\{\frac{ta_s}{a_1}\}$ denotes the fractional part of $\frac{ta_s}{a_1}$. Let $a_0=1+4l(n-1),2n-1+4k(n-1)$, $a_1=2n-1+4k(n-1)$ and $a_2=4(n-1)$. It is straightforward to check that $\mld(P_2)=\frac{1}{n-1}$. To compute $\operatorname{mld}(P_1)$ write the set of indices $I=\{1,\dots,a_1-1\}$ in the form $I=\{ta_2^{-1} \pmod{a_1}:t\in I\}$ where $a_2^{-1}$ denotes the inverse of $a_2$ modulo $a_1$. Hence, $\operatorname{mld}(P_1)=\operatorname{min}_{1\leq t\leq a_1-1}\{\frac{t\alpha}{a_1}\}+\frac{t}{a_1},$ where $\alpha=a_2^{-1}a_0$.
For a positive integer $T<a_1$, we observe that the values $\{\frac{t\alpha}{a_i}\}$ with $0\leq t\leq T$ are all distinct and contained in $[0,1)$. It follows that there exist some $t<t^\prime$ such that $\{\frac{t\alpha}{a_i}\}$ and $\{\frac{t^\prime\alpha}{a_i}\}$ lie in an interval of length $\frac{1}{T}$. Therefore, if $t_0=t^\prime-t$, then $\{\frac{t_0\alpha}{a_1}\}<\frac{1}{T}$. Choosing $T=2(n-1)$ yields that $$\left\{\frac{t_0\alpha}{a_1}\right\}+\frac{t_0}{a_1}<\frac{1}{2(n-1)}+\frac{2(n-1)}{a_1}.$$ If $k\geq n-1$, then $a_1> 4(n-1)^2$ and therefore $\operatorname{mld}(P_1)<\frac{1}{n-1}$. It follows that the singularity attaining $\operatorname{mld}(\mathcal{X}_t)$ is not $P_2$. The bounds of the other singularities follow from the value of the mld of the singularity $P_2$, keeping the the same numbering as in (1). Indeed, $\operatorname{mld}(P_2)=\frac{5}{6n-5}$ in (2) and $\operatorname{mld}(P_2)=\frac{5}{6n-7}$ in (3).
\end{proof}
Now we proceed to show that the set of weighted projective planes containing a Du Val point and those presented in Corollary \ref{weightedlist} have density zero in $\mathbb{Z}_{>0}^3$. Here, we use the natural density, see for example \cite[Chapter III]{Tenenbaum15}.

\begin{lemma}
Let $S \subset \mathbb{Z}_{>0}^3$ be the union of the families (up to permutation):
\begin{enumerate}
\item the set \(\mathcal{A}\) of all well-formed triples \((a,b,c)\) for which the weighted projective plane \(\mathbb{P}(a,b,c)\) contains either Du Val singularities or a torus invariant smooth point.
\item $\mathcal{B}_1 = \{ \bigl(1+4\ell(n-1),2n-1+4k(n-1),4n-4\bigr)\mid n\geq 2,\, 0 \le \ell,k < n-1\}$,
\item  $\mathcal{B}_2 = \{ \bigl(1+\ell(6n-5),\,3n-1+k(6n-5),\,6n-5\bigr) \mid  n \geq 2,\, 0 \leq \ell,k < \lceil\tfrac{4(6n-5)}{9}\rceil \}$,
\item $\mathcal{B}_3 = \{ \bigl(1+\ell(6n-7),\,3n-2+k(6n-7),\,6n-7\bigr) \mid n \geq 2,\, 0 \leq \ell,k < \lceil\tfrac{4(6n-7)}{9}\rceil \}$.
\end{enumerate}
Then \(S\) has density zero in \(\mathbb{Z}_{>0}^3\); equivalently,
$$\lim_{N \to \infty} \frac{\#(S \cap [1,N]^3)}{N^3} = 0.$$
\label{densitylem}
\end{lemma}
\begin{proof}
For the family \textit{(1)}, we observe that the singular point $P=\frac{1}{a}(b,c)$ in $\mathbb{P}(a,b,c)$ is an $A_{a-1}$ singularity if and only if $b+c\equiv 0 \pmod{a}$. This also includes the possible non-singular torus invariant points. Thus, For a fixed $a$, this gives $$\#\{(b,c)\in [1,N]^2:b+c\equiv 0\pmod{a}\}\leq \sum_{r=0}^{a-1}\Big\lceil\frac{N}{a}\Big\rceil^2\leq a\Big(\frac{N}{a}+1\Big)^2.$$ Consequently, $$\sum_{a\leq N}\#\{(b,c)\in [1,N]^2:b+c\equiv 0\pmod{a}\}\leq N^2\sum_{a\leq N}\frac{1}{a}+O(N^2)\ll N^2\operatorname{log}(N)+O(N^2).$$ Summing similarly over all choices of $b,c$, we obtain, $\lim_{N \to \infty} \frac{\#(\mathcal{A} \cap [1,N]^3)}{N^3} = \lim_{N \to \infty}\frac{3N^2\operatorname{log}(N)+O(N^2))}{N^3}=0$. 

For the family \textit{(2)}, a triple $(a,b,c) \in \mathcal{B}_1$ satisfies $c = 4n-4 \le N$, so $n \le \lfloor N/4 \rfloor + 1$. For each such $n$, the admissible values of $\ell$ and $k$ satisfy $0 \le \ell < n-1$, $0 \le k < n-1$, and the constraints $a = 1 + 4l(n-1)\leq N$, $b =2n-1+4k(n-1)\leq N$.  

Hence, for a fixed $n$, the number of admissible triples is $$T(n) := \#\{ (\ell,k) : 0 \leq \ell,k < n-1,\, a \leq N,\, b \leq N \}\leq \frac{N^2}{16 (n-1)^2}.$$
We now sum over all $2 \leq n \leq \lfloor N/4 \rfloor + 1$. We make this by splitting the sum at $n_0 =\lfloor \sqrt{N} \rfloor$:  

\begin{enumerate}
\item For $2 \leq n \leq n_0$, it follows that $T(n)\leq n^2 \le N$. Hence $$\sum_{n=2}^{n_0} T(n) \le \sum_{n=1}^{\sqrt{N}} n^2 = \frac{\sqrt{N}(\sqrt{N}+1)(2\sqrt{N}+1)}{6} \le N^{3/2}.$$

\item For $n_0 < n \le \lfloor N/4 \rfloor + 1$, we use the previous bound $T(n) \le \frac{N^2}{16 (n-1)^2} \le \frac{N^2}{16 n^2}$. Then $$\sum_{n=n_0+1}^{\lfloor N/4 \rfloor + 1} T(n) \le \frac{N^2}{16} \sum_{n > \sqrt{N}} \frac{1}{n^2} \le \frac{N^2}{16} \cdot \frac{1}{\sqrt{N}-1} < N^{3/2}.$$
\end{enumerate}
Combining both ranges gives $\#(\mathcal{B}_1 \cap [1,N]^3)< 6 N^{3/2}$. Thus, we deduce that the density of $\mathcal{B}_1$ is given by $\lim_{N \to \infty}\frac{\#(\mathcal{B}_1 \cap [1,N]^3)}{N^3} =\lim_{N \to \infty} \frac{6 N^{3/2}}{N^3}=0$. An analogous statement holds for the families \textit{(2)} and \textit{(3)}. Consequently, the set $S$ has density $0$ as well.
\end{proof}

\begin{proof}[Proof of Theorem~\ref{weighteddeg}]
Let $\mathcal{X}\rightarrow \mathbb{D}$ be a klt degeneration 
of $\mathcal{X}_t \simeq \pp(a,b,c)$ where $(a,b,c)\in \zz_{\geq 1}^3$ is a well-formed triple. Here, we mean that the pair 
$(\mathcal{X},\mathcal{X}_0)$ is a plt pair. 
By Corollary~\ref{weightedlist} and Lemma~\ref{densitylem} 
there exists a subset $S\subsetneq \zz_{\geq 1}^3$ 
satisfying the following conditions:
\begin{enumerate}
\item The subset $S$ has density zero in the natural density endowed from $\zz^3$; and 
\item for every $(a,b,c)\in \zz_{\geq 1}^3 \setminus S$ the klt degeneration $\mathcal{X}\rightarrow \mathbb{D}$ admits a $1$-complement $(\mathcal{X},\mathcal{B}+\mathcal{X}_0)$. 
\end{enumerate} 
By Lemma~\ref{densitylem}, (1) every torus invariant point 
of $\mathcal{X}_t$ has minimal log discrepancy $<1$. 
Therefore, for $t\in \mathbb{D}$ general, 
the boundary divisor $\mathcal{B}_t$ must pass through every torus invariant point of $\mathcal{X}_t$ and be nodal at such points. 
This implies that $\mathcal{B}_t$ is just the toric boundary
of $\mathbb{P}(a,b,c)$.
We argue that $\mathcal{X}_0$ is a toric surface of Picard rank one
with the same singularities than $\mathbb{P}(a,b,c)$. 
The fact that $\mathcal{X}_0$ has Picard rank one follows from \cite[Lemma 2]{Man91}. Further, from Corollary \ref{cor:const-sing}, we know that $\mathcal{X}_0$ has the same singularities as $\mathbb{P}(a,b,c)$ and possibly some extra Wahl singularities. 
If $\mathcal{X}_0$ has some extra Wahl singularities, then the 
$1$-complement $\mathcal{B}_0$ must pass through at least $4$ singular points, and it must be nodal at such points. 
This would imply that $\mathcal{B}_0$ has at least $4$ components. 
Thus, we would get $c(\mathcal{X}_0,\mathcal{B}_0)<0$ leading to a contradiction. 
Hence, we conclude that $\mathcal{X}_0$ has the same three singular points than $\mathbb{P}(a,b,c)$ and that $(\mathcal{X}_0,\mathcal{B}_0)$ is a toric pair.

We argue that $\mathcal{B}$ has trivial monodromy over $\mathbb{D}$, i.e., every component of $\mathcal{B}_t$ is indeed the restriction
of an irreducible component of $\mathcal{B}$. 
Assume otherwise and take the finite cover given by Lemma~\ref{lem:finite-cover}, so we obtain a commutative diagram as follows: 
\[
\xymatrix{
(\mathcal{X},\mathcal{X}_0+\mathcal{B}) \ar[d]_-{\pi} & 
(\mathcal{X}',\mathcal{X}'_0+\mathcal{B}') \ar[d]^-{\pi'} \ar[l]_-{f} \\
\mathbb{D} & \mathbb{D} \ar[l]_-{f_\mathbb{D}}
}
\]
where $f$ is a crepant finite Galois morphism.
This induces a crepant finite Galois morphism between central fibers 
\[
f_0\colon  
(\mathcal{X}_0',\mathcal{B}_0') \rightarrow (\mathcal{X}_0,\mathcal{B}_0).
\]
As $(\mathcal{X}_0,\mathcal{B}_0)$ is a toric pair, 
and $f_0$ is unramified on the complement of $\mathcal{B}_0$, 
we conclude that $f_0$ is a toric cover (see, e.g.,~\cite[Proposition 3.32]{Mor20c}.
Thus, the Galois group acts trivially on 
$\mathcal{D}(\mathcal{B}_0')$. 
This contradicts the fact that $\mathcal{B}$ has monodromy over $\mathbb{D}$. 
Therefore, we conclude that every component of $\mathcal{B}_t$ is the restriction of a component of $\mathcal{B}$.
Thus, we get 
\[
c_0(\mathcal{X}/\mathbb{D},\mathcal{X}_0+\mathcal{B})=
\dim \mathcal{X}  +\rho(\mathcal{X}) - |\mathcal{X}_0+\mathcal{B}| =0.
\]
By~\cite[Theorem 1.1]{MS21}, we conclude that 
$(\mathcal{X},\mathcal{X}_0+\mathcal{B})\rightarrow \mathbb{D}$ is a formally toric morphism near $\{0\}$. 
Hence, this morphism is simply a product if and only if the central fiber is reduced. 
This finishes the proof. 
\end{proof} 

\begin{proof}[Proof of Corollary~\ref{cor:markov-no-deg}]
Note that the condition $(a,b,c)\in \zz_{\geq 2}$
implies that $\pp(a^2,b^2,c^2)$ has three singular points
with orbifold indices $a^2,b^2,c^2$, respectively.
The mld of such points are $\frac{1}{a}$, $\frac{1}{b}$, 
and $\frac{1}{c}$, respectively. 
Thus, ${\rm mld}(\pp(a^2,b^2,c^2)) < \frac{1}{6}$ 
by the classification of Markov triples. 
None of the sets $\mathcal{B}_1,\mathcal{B}_2, \mathcal{B}_3$
in Lemma~\ref{densitylem} contain Markov triples.
Hence, the statement follows from Theorem~\ref{weighteddeg}. 
\end{proof}

\section{Degenerations of weighted projective planes.}
\label{sec:degen-wps}

In this section, we focus on the proof of Theorem~\ref{thm:1-1-n}. Along the lines of Theorem \ref{thm:1-comp}, we first show that for any klt Fano degeneration $\pi:\mathcal{X}\to\mathbb{D}$ with general fiber $\mathcal{X}_t\cong \pp(1,1,n)$, the central fiber $\mathcal{X}_0$ admits a 1-complement for all $n\geq 2$. We then classify all possible limits $\mathcal{X}_0$ when $\mathcal{X}_0$ is a toric surface. As a consequence, we obtain a full classification of the degenerations of $\pp(1,1,n)$ with $n\geq 3$ and reduced central fiber $\mathcal{X}_0$. \newline
Throughout this section, we denote by $\mathbb{F}_n$ the Hirzebruch surface whose negative section $\Delta_0$ satisfies $\Delta_0^2=-n$.
By Lemma \ref{lem:min-resol-horizontal}, any degeneration of $\mathbb{P}(1,1,n)$ with $n\geq 2$ can be described in terms of a degeneration of $\mathbb{F}_n$, after extracting a divisor $E$ over the family. Following \cite[Theorem 1]{B85}, if $X$ is a normal projective surface with klt singularities that admits a $\mathbb{Q}$-Gorenstein smoothing to some $\mathbb{F}_n$, then $X$ must be a rational surface. Moreover, if $\phi:Y\to X$ denotes its minimal resolution, there exists a maximum integer $d>0$ such that $Y$ is obtained by a sequence of blow-ups away from the section $\Delta_0\subset\mathbb{F}_d$. The number $d$ is defined to be the {\em weight} of $X$, see \cite[Definition 2]{Man91}. The corresponding birational morphism $\mu:Y\to \mathbb{F}_d$ induces a genus zero fibration $Y\to\pp^1$. \newline
The following lemma was previously proved by
the second author in \cite[Theorem 5.6]{MNU24}. We rewrite the statement and include a proof here for the reader's convenience.
\begin{lemma}\label{lem:Hz-minres}
Let $\pi:\mathcal{X}\to\mathbb{D}$ be a $\mathbb{Q}$-Gorenstein degeneration of $\mathbb{F}_n$. Assume that $X=\mathcal{X}_0$ is a singular normal projective surface with klt singularities. Let $\phi:Y\to X$ be the minimal resolution of $\mathcal{X}_0$, and let $\mu:Y\to \mathbb{F}_d$ be the associated minimal model of $Y$ where $d$ is the weight of $X$. Then, $\mu(Exc(\phi))$ consists of the section $\Delta_0$ with at most two fibers $F_1,F_2$.
\end{lemma}
\begin{proof}
Assume that $\rho(X)=2$, the case $\rho(X)=1$ follows from \cite[Theorem 7.6]{HP10}. By Lemma \ref{lem:sing-central-fiber} we observe that $X$ contains only Wahl singularities. 
Let $E=\operatorname{Exc}(\phi)$ denote the exceptional divisor of $\phi$, and let $h$ be the number of degenerated fibers of $\mu$.
First, note that $E$ cannot be fully contained in $\operatorname{Supp}(\sum_{i=1}^h f_i)$. Otherwise, $E$ would contain every component in $f_i$ except for a unique $(-1)$ curve. Indeed, by \cite[Proposition 7.4]{HP10}, the curve $\phi(f_i)$ would then meet two singularities, one of which is not Wahl, a contradiction.
We now proceed to bound $h$. Let $e$ denote the number of $(-1)$ curves in $Y$, and abusing notation, let $\Delta_0$ be the strict transform of $\Delta_0\subset \mathbb{F}_d$ under $\mu$. It follows that $e \leq h+1$. Otherwise, we would obtain  $b_2(\overline{E \setminus\Delta_0})+e>\sum_{i=1}^h b_2(f_i),$ which implies that $E$ contains an exceptional curve different from $\Delta_0$ and not contained in any degenerated fiber. However, this contradicts \cite[Proposition 9]{Man91}, since $h^0(-K_Y)=h^0(-K_X)\geq h^0(-K_{\mathbb{F}_n})+1=9$. 
Hence, we are left with two possibilities for the degenerate fibers:
\begin{enumerate}
\item $e=h$ : All degenerated fibers contain exactly one $(-1)$ curve.
\item $e=h+1$ : There exists a unique degenerated fiber containing two $(-1)$ curves, say $f_1$, while each remaining $f_i$ contains exactly only one $(-1)$ curve.
\end{enumerate}
In \emph{Case (1)}, the equality $b_2(\overline{E \setminus\Delta_0})+e+1=\sum_{i=1}^h b_2(f_i)$ implies there exists an irreducible curve $C$ in a degenerate fiber $f_j$, such that $C\not\subset\operatorname{Supp}(E)$ and $C^2\leq -2$. We relabel this fiber $f_1$. Let $E_1$ denote the $(-1)$ curve of $f_1$. Note that neither $E_1$ nor $C$ intersect $\Delta_0$. The claim for $E_1$ follows immediately; otherwise $f_1$ must contain an additional $(-1)$ curve. If $C$ were to intersect $\Delta_0$, then by \cite[Proposition 7.4]{HP10}, the fiber $f_1$ would have a curve configuration of the form $(-e_1)-\cdots-(-e_k)-(-1)-(-b_1)-\cdots-(-b_l)-(-c)$, where $C^2=-c$ and $e_i,b_i\geq 2$ for all $i$. Since $X$ carries only Wahl singularities, it follows that $\sum_{i=1}^k e_i=3k+1$ and $\sum_{i=1}^l b_i=3l+1$, which together imply that $c=1$, a contradiction. On the other hand, we note that $\overline{f_1\setminus{(E_1+C)}}$ has at most three connected components. \newline
In \emph{Case (2)}, the only situation where a $(-1)$ curve in $f_1$ intersects $\Delta_0$ is when $f_1$ consists of a single blowup on a generic point of $\mathbb{F}_d$. Otherwise, $f_1$ would have a configuration of curves $(-1)-(-2)-\cdots-(-2)-(-1)$, implying that $X$ contains an $A_k$ singularity, a contradiction. Analogously to the previous case, the complement $\overline{f_1\setminus{(E_1+E_2)}}$ has at most three connected components. \newline
Since the connected component that contains $\Delta_0$ is a chain of curves, we obtain that $h \leq 3$ and only attaining equality when $f_1$ is obtained by a single blow-up on a generic point of $\mathbb{F}_d$. Finally, let $E_i^{\prime}$ denote the $(-1)$ curves in the remaining fibers. In each case, $\overline{f_i\setminus E_i^{\prime}}$ contains at most 2 connected components. We conclude that $\pi(E)$ consists of $\Delta_0$ and at most two fibers $f_1,f_2$.
\end{proof}

\begin{proposition}\label{prop:1-comp 1-1-n}
Let $\mathcal{X} \to \mathbb{D}$ be a klt Fano degeneration of $\mathbb{P}(1,1,n)$ with $n\geq 2$. Then, the pair $(\mathcal{X},\mathcal{X}_0)$ admits a $1$-complement over $\mathbb{D}$.
\end{proposition}
\begin{proof}
We proceed as in the proof of Theorem \ref{thm:1-comp}. Following \emph{Step 1}, let $(\mathcal{E}\subset \mathcal{Y})\to\mathbb{D}$ be the family obtained by extracting the $(-n)$ curve over the point $\frac{1}{n}(1,1)$. As in \emph{Step 2}, let $C$ be the extremal curve $C$ such that $(K_{\mathcal{Y}_0}+dE_0)\cdot C=0$, where $d>0$ is the anti-nef threshold of $(\mathcal{Y}_0,dE_0)$. By Lemma \ref{lem:Hz-minres}, in the minimal resolution $Y\to\mathcal{Y}_0$, the strict transform of $C$ corresponds to the unique $(-1)$ curve contained in the degenerated fiber $f_1$ of the associated fibration $Y\to\pp^1$.
As in \emph{Step 3}, let $\varphi:\mathcal{Y}_0\to\mathcal{Y}_0^\prime$ be the contraction induced by $C$. By Lemma \ref{lem:Hz-minres}, $C$ and $E_0$ intersect in a single point, therefore after contracting $C$ under $\varphi$, the pair  $(\mathcal{Y}_0^\prime,E_0^\prime)$ is plt.
In this situation, there can be at most one singular point of $\mathcal{Y}_0$ lying on $E_0$ outside $E_0\cap C$. Consequently, $E_0^\prime$ contains at most two singular points of $\mathcal{Y}_0^\prime$. Finally, following \emph{Step 5}, we obtain a 1-complement of the pair $(\mathcal{X},\mathcal{X}_0)$ over $\mathbb{D}$.
\end{proof}
Combining the previous proposition with Corollary \ref{cor:const-sing}, we deduce that for $n\geq 3$, the central fiber in any klt Fano degeneration of $\pp(1,1,n)$ must contain a point of the form $\frac{1}{n}(1,1)$. 
\begin{proposition}\label{prop:sing-limits_1-1-n}
Let $\mathcal{X} \to \mathbb{D}$ be a klt Fano degeneration of $\mathbb{P}(1,1,n)$ with $n\geq 3$. Then the cyclic quotient singularity $P=\frac{1}{n}(1,1)$ on the general fiber specializes to a cyclic quotient singularity of the same type on the central fiber. Moreover, the singularity basket of $\mathcal{X}_0$ consists of a point $\frac{1}{n}(1,1)$ and at most two Wahl singularities different from $P$.
\end{proposition}
\begin{proof}
Since $\mathcal{X}_0$ is 1-complemented, Corollary \ref{cor:const-sing} implies that the singularity basket of $\mathcal{X}_0$ coincides with that of $\pp(1,1,n)$, possibly with additional Wahl singularities. As $n\geq 3$, a 1-complement $\mathcal{B}_0\in |-K_{\mathcal{X}_0}|$ intersects every singular point of $\mathcal{X}_0$ and is nodal at such points. Therefore, $c(\mathcal{X}_0,\mathcal{B}_0)\geq 0$ if and only if $\mathcal{X}_0$ contains at most two Wahl singularities different from $P$.
\end{proof}
We are in a position to classify all the possible limits of klt Fano degenerations of $\pp(1,1,n)$ where $\mathcal{X}_0$ is a toric surface. For this part, we make a similar proof to \cite[Theorem 4.1]{HP10}.

\begin{lemma}\label{lem:mutations} 
Let $n \geqslant 1$. Any solution of the equation $n+x^2+y^2=(n+2)xy$ over the positive integers is obtained from the mutation process $(x, y) \mapsto(y,(n+2) y-x)$ starting from $(1,1)$. In particular, it follows that $\operatorname{gcd}(n,x,y)=1$.
\end{lemma}
\begin{proof}
Let $(x,y) \in\mathbb{Z}_{\geq 1}^2$ satisfy $n+x^2+y^2=(n+2) x y$. Then $(x,(n+2)x-y)$ is also a solution, and since $(n+2) x-y=\frac{x^2+n}{y}>0$, the mutation preserves positivity. The only solution with $x=y$ is $(1,1)$. Assume now $1 \leqslant x<y$. We have $y^2>x^2+n$. Since the inequality is equivalent to $(n+2) x y>2\left(x^2+n\right)$, the condition $y \geqslant x+1$ gives it automatically. The inequality $y^2>x^2+n$ gives $y>y^{\prime}=(n+2)x-y$. Thus, each mutation strictly decreases the larger coordinate while keeping all entries positive. Iterating the process eventually yields $(1,1)$, and every solution arises in this way from successive mutations of $(1,1)$. Inductively, we conclude that $\operatorname{gcd}(n, x, y)=1$ for every solution $(x,y)$ of the equation.
\end{proof}

\begin{proposition}\label{prop:tor-deg}
Let $\mathcal{X}\rightarrow \mathbb{D}$ be a klt Fano degeneration of $\pp(1,1,n)$ with $n\geq 3$, then $\mathcal{X}_0$ is toric if and only if $\mathcal{X}_0\cong \pp(x^2,y^2,n)$ such that $x^2+y^2+n=(n+2)xy$.
\end{proposition}
\begin{proof}
Suppose that $X\cong\pp(n,x^2,y^2)$, where the weights satisfy the equality in the statement. The surface $X$ contains singular points
$$
P_0=\frac{1}{n}(1,1), P_1=\frac{1}{n_0^2}\left(1, n_0 w_{n_0}-1\right) \text { and } P_2=\frac{1}{n_1^2}\left(1, n_1 w_{n_1}-1\right)
$$
where $w_{n_0}\equiv(n+2)n_1^{-1}\pmod{n_0}$ and $w_{n_1}\equiv(n+2) n_0^{-1}\pmod{n_1}$. Since $X$ is klt and $-K_X$ is ample, by \cite[Proposition 3.1]{HP10} it follows that $H^2(X, T_X)=0$, hence $X$ has no local to global obstructions to deform. In particular, it admits a $\qq$-Gorenstein deformation to a del Pezzo surface $X^{\prime}$ with $\rho\left(X^{\prime}\right)=1, K_X^2=\frac{(n+2)^2}{n}$ and unique singular point $P=\frac{1}{n}(1,1)$. Its minimal resolution $Y \rightarrow X^{\prime}$ satisfies $K_Y^2=8$ and consequently $Y\cong\mathbb{F}_n$. Therefore, it must hold that $X^{\prime}\cong\mathbb{P}(1,1,n)$.

Conversely, let $\mathcal{X} \rightarrow \mathbb{D}$ be a klt Fano degeneration of $\mathbb{P}(1,1,n)$ with $n \geq 3$, and assume that $\mathcal{X}_0$ is toric. By Proposition \ref{prop:sing-limits_1-1-n}, the surface $\mathcal{X}_0$ contains a quotient singularity of type $\frac{1}{n}(1,1)$ and possibly some Wahl singularities with Gorenstein indices $n_0$ and $n_1$. Let $p:\pp(w_0,w_1,w_2)\to\mathcal{X}_0$ be the finite covering induced by the inclusion of lattices $N_{\mathcal{X}_0}\subset N=\mathbb{Z}^2$ of index $\left[N: N_{\mathcal{X}_0}\right]=e$. The singularities of $\mathcal{X}_0$ are then obtained as quotients of the singularities of $\pp(w_0, w_1, w_2)$. 
Without loss of generality, we may assume that $e w_i=n_i^2$ for $i=0,1$ and $e w_2=n$. Since $K_{\mathbb{P}\left(w_0, w_1, w_2\right)}^2=e K_{\mathcal{X}_0}^2$, we derive that $n_0^2+n_1^2+n=(n+2) n_0 n_1$. Since $e=\operatorname{gcd}\left(n,n_0^2, n_1^2\right)$, Lemma \ref{lem:mutations} implies that $e=1$. Therefore, $\mathcal{X}_0\cong\pp(n,n_0^2,n_1^2)$ where $n_0^2+n_1^2+n=(n+2)n_0n_1$.
\end{proof}

\begin{proposition}\label{prop:min-res-1-1-n}
Let $\mathcal{X}\to\mathbb{D}$ be a klt Fano degeneration of $\pp(1,1,n)$ with $n\geq 3$ such that $\mathcal{X}_0$ contains a singular point $P\neq\frac{1}{n}(1,1)$. Let $\phi:Y \rightarrow \mathcal{X}_0$ be the minimal resolution of $\mathcal{X}_0$ and let $d$ be the weight of $d$. Then $d\in\{n+3,n+6\}$.\newline Under the morphism $\mu:Y\to\mathbb{F}_d$, it follows that $\mu(\rm{Exc}(\phi))$ consists of the section $\Delta_0$, a fiber $F_1$ and possibly a fiber $F_2$.
Let $f_i$ denote the strict transforms of $F_i$ for $i=1,2$. Then $f_1$ consists of a chain of rational curves $(-n)-(-1)-(-2)-\cdots-(-2)$, while the possible $f_2$ satisfies that $\overline{f_2\setminus(\rm{Exc}(\phi)\cap f_2)}$ is a unique $(-1)$ curve.
\end{proposition}
\begin{proof}
Let $\mathcal{X}\rightarrow \mathbb{D}$ be a Fano degeneration such that $\mathcal{X}_t=\pp(1,1,n)$ for $t\neq 0$. By Lemma \ref{extraction}, there exists a family $\pi:(\mathcal{E}\subset \mathcal{Y})\to\mathbb{D}$ together with a divisorial contraction $\pi:\mathcal{Y}\to \mathcal{X}$ contracting $\mathcal{E}$, where for each $t$, we have $\mathcal{E}|_{\mathcal{Y}_t}^2=-n$. In particular $\mathcal{Y}_t=\mathbb{F}_n$ for $t\neq 0$. By Proposition \ref{prop:sing-limits_1-1-n}, $\mathcal{E}|_{\mathcal{Y}_0}$ must contract to the point $P=\frac{1}{n}(1,1)$. 
Following \cite[Section 1.3.2]{BC94}, since $K_{\mathcal{Y}_0}$ is $\pi|_{\mathcal{Y}_0}$-ample then $E_0:=\mathcal{E}|_{\mathcal{Y}_0}$ must be a $(-n)$ curve.\newline
Let $Y\to \mathcal{Y}_0$ be the minimal resolution of $\mathcal{Y}_0$. By Lemma \ref{lem:Hz-minres}, it follows that $\mu(\rm{Exc}(\phi))$ consists of the section $\Delta_0$, a fiber $F_1$ and possibly a fiber $F_2$. Denote by $\widetilde{E_0}$ of $E_0$, then by Lemma \ref{lem:Hz-minres} $\widetilde{E_0}$ is contained in the degenerate fiber $f_1$ of $Y\to \pp^1$. Moreover, if $C$ corresponds to the unique $(-1)$ curve of $f_1$, then $C$ intersects $\widetilde{E_0}$ transversally and, by \cite[Proposition 7.4]{HP10} $f_1$ consists of the curve configuration: $(-n)-(-1)-(-2)-\cdots-(-2)$, where the chain of $(-2)$ curves has length $n-1$. If $Y\to\pp^1$ contains an additional degenerate fiber $f_2$, by Lemma \ref{lem:Hz-minres} it follows that $\overline{f_2\setminus(\rm{Exc}(\phi)\cap f_2)}$ is a unique $(-1)$ curve.\newline
We now compute the weight of $\mathcal{X}_0$, denoted by $d\ge 0$.
Let $P=\frac{1}{n_0^2}(1,n_0a_0-1)$ be the singular point of $\mathcal{X}_0$ whose dual graph contains the strict transform of $\Delta_0$ under the morphism $Y\to\mathbb{F}_d$. By \cite[Proposition 2.5]{UZ24}, we obtain that $$\frac{n_0^2}{n_0a_0-1}=[2,\cdots,2,d,W_1^\vee],$$ where this notation refers to concatenation of the corresponding sequences and ${W_1^\vee}$ denotes the Hirzebruch-Jung continued fraction of some rational number $\frac{n_1^2}{n_1a_1+1}$ where $0<a_1<n_1$ and $\operatorname{gcd}(n_1,a_1)=1$. The sequence $W_1^\vee$ may also be empty. If $l$ denotes the length of the chain of $P$ and $S$, the sum of its entries (analogously $l_{W_1^\vee}, S_{W_1^\vee}$ for $W_1^\vee$), it follows that $S=3l+1$ and $S_{W_1^\vee}=3l_{W_1^\vee}-3$ (see for example \cite[Section 4]{UZ24}). Consequently, $d=n+6$ if $W_1^\vee$ is not empty, and $d=n+3$ if $W_1^\vee$ is empty.
\end{proof}

\begin{theorem}\label{thm:partialsmoothing}
Let $\mathcal{X}\rightarrow \mathbb{D}$ be a klt Fano degeneration of $\pp(1,1,n)$ with $n\geq 3$. Then, the central fiber $\mathcal{X}_0$ is a partial $\mathbb{Q}$-Gorenstein smoothing of a weighted projective plane $\pp(x^2,y^2,n)$ satisfying $x^2+y^2+n=(n+2)xy$.
\end{theorem}
\begin{proof}
Assume that $\mathcal{X}_0$ is not toric. Let $\phi:Y\to\mathcal{X}_0$ be the minimal resolution of $\mathcal{X}_0$. By Proposition \ref{prop:min-res-1-1-n}, the induced fibration $Y\to\pp^1$ must contain degenerate fibers $f_1$ and $f_2$. Let $E_2\subset Y$ be the $(-1)$ curve in the minimal resolution of $\mathcal{X}_0$, then $E=\phi_*(E_2)$ does not form a toric coordinate axis near the point $P$, the singular point whose minimal resolution containing the strict transform of $\Delta_0$ under $\mu:Y\to \mathbb{F}_d$. Indeed, $E_2$ arises as a non-toric blow-up on a component of $f_2$ and $\overline{f_2\setminus E_2}$ consists of a single chain of rational curves. Following \cite[Definition 4.6]{UZ25}, we observe that the curve $E_2$ admits a \textit{right slide}. Since $-K_{\mathcal{X}_0}$ is an ample divisor, by \cite[Proposition 3.1]{HP10}, it follows that $H^2(\mathcal{X},T_{\mathcal{X}_0})=0$.
By \cite[Lemma 5.11]{UZ25}, we deduce that there exists a $\qq$-Gorenstein deformation $\pi^\prime:\mathcal{X}^\prime\to \mathbb{D}$,  where $\mathcal{X}_t^{\prime}\cong\mathcal{X}_0$ for $t \neq 0$ and $E$ degenerates to an irreducible curve $E^{\prime}$ forming a toric coordinate axis near $P$ and near a new Wahl singularity $P^\prime=\frac{1}{n_1^2}\left(1, n_1 a_1-1\right)$. This deformation is obtained by globalizing a deformation on a neighborhood of $E^{\prime}$. \newline
Let $\mathcal{X}_0^\prime$ be the central fiber of $\pi^\prime$ and $\phi^\prime:\mathcal{Y}_0^\prime\to \mathcal{X}_0^\prime$ its minimal resolution. In this new situation, the chain of curves resolving $P^\prime$ lies entirely in a degenerate fiber $f_2$ of $Y^\prime\to\pp^1$. Let $E_1$ denote the $(-1)$ curve of the fiber $f_1$ as in Proposition \ref{prop:min-res-1-1-n}, and let $C$ be the strict transform of a general section $s\sim\Delta_0+df\subset\mathbb{F}_d$. Then $\mathcal{B}=\phi^\prime_*(C)+\phi^\prime_*\left(E_1\right)+E^{\prime}$ forms a toric boundary for $\mathcal{X}_0^{\prime}$. The statement follows from Theorem \ref{prop:tor-deg}.
\end{proof}
\begin{proof}[Proof of Theorem~\ref{thm:1-1-n}]
Let $\mathcal{X}\rightarrow \mathbb{D}$ be a klt Fano degeneration of $\pp(1,1,n)$ with $n\geq 3$, assuming that $\mathcal{X}\neq \pp(1,1,n)$. By Theorem \ref{thm:partialsmoothing}, it remains to show that any non-toric central fiber $\mathcal{X}_0$ admits an effective $\mathbb{G}_m$-action. By Proposition~\ref{prop:1-comp 1-1-n}, $\mathcal{X}_0$ admits a $1$-complement. Hence, by \cite[Proposition 6.1]{P15}, it carries an effective $\mathbb{G}_m$-action. Furthermore, Theorem \ref{thm:partialsmoothing} implies that for $n\geq 3$, a non-toric $\mathcal{X}_0$ has exactly two non-Gorenstein points, so the conditions of \cite[Proposition~6.1]{P15} are satisfied.
\end{proof}

\section{Examples and Questions}
\label{sec:ex-and-quest}

In this section, we present examples and questions for further research.

\begin{example}\label{ex:non-toric-central}
{\em
We give an example of a degeneration of a toric pair where the central fiber is not a toric variety. 
Fix $n\geq 3$ and consider the variety $\mathcal{Y}:=(\pp^1)^n \times \mathbb{D}$ 
and let $\pi_2\colon \mathcal{Y}\rightarrow \mathbb{D}$ be the projection onto the second component. 
Let $B_T$ be the torus-invariant boundary of $(\pp^1)^n$ 
and $\mathcal{B}_\mathcal{Y}:=B_T\times \mathbb{D}$.
Therefore, the pair $(\mathcal{Y},\mathbb{B}_{\mathcal{Y}}+\mathcal{Y}_0)$ 
is a log Calabi--Yau pair for which every fiber of $\pi_2$ is isomorphic to
$((\pp^1)^n,B_T)$. 
Consider the group $\zz_2$ acting on $\mathcal{Y}$ via 
\begin{align*} 
\mu & \colon  (\pp^1)^n\times \mathbb{D} \rightarrow 
(\pp^1)^n \times \mathbb{D} \\
\mu & \cdot ([x_1:y_1],[x_2:y_2],\dots,[x_n:y_n],t) :=
([y_1:x_1],[y_2:x_2],\dots,[y_n:x_n],-t).
\end{align*} 
The log Calabi--Yau pair $(\mathcal{Y},\mathcal{B}_{\mathcal{Y}}+\mathcal{Y}_0)$
is invariant under the $\zz_2$-action. 
Let $(\mathcal{X},\mathcal{B}+\mathcal{X}_0)$ be the induced quotient, so
we obtain a commutative diagram:
\[
\xymatrix{
(\mathcal{X},\mathcal{B}+\mathcal{X}_0) \ar[d]^-{\pi}  & (\mathcal{Y},\mathcal{B}_{\mathcal{Y}}+\mathcal{Y}_0) \ar[d]^-{\pi_2} \ar[l]^-{/\zz_2} \\
\mathbb{D} & \mathbb{D} \ar[l]^-{/\zz_2}
}
\]
The morphism $\pi$ is a degeneration of the toric log Calabi--Yau pair
$((\pp^1)^n,B_T)$ into $(\mathcal{X}_0,\mathcal{B}_0)=((\pp^1)^n,B_T)/\zz_2$. 
We have that $\mathcal{D}(\mathcal{X}_0,\mathcal{B}_0)\simeq_{\rm PL} \mathbb{P}_\rr^{n-1}$. 
Thus, the central fiber $(\mathcal{X}_0,\mathcal{B}_0)$ is not a toric pair. 
}
\end{example} 

\begin{example}\label{ex:P^2-toric-model-1}
{\em 
We consider the pair
$(\pp^2,B)$ where $B$ is the sum of a quadric and a transversal line. 
This example is of cluster type, as described by Gross, Hacking, and Keel~\cite{GHK15}. 
The choice of a torus embedding $\mathbb{G}_m^2\hookrightarrow \pp^2\setminus B$ is equivalent to the choice of a {\em toric model} 
$\phi \colon (T,B_T)\dashrightarrow (\pp^2,B)$. 
More precisely, $\phi$ is a crepant birational map that only extracts log canonical places and $(T,B_T)$ is a toric pair. In this example, we demonstrate 
that the choice of such a toric model, or equivalently the choice of an embedding $\mathbb{G}_m^2\hookrightarrow \pp^2\setminus B$ naturally induces a degeneration of $\pp^2$ to a toric pair.
We will consider a toric model
$(\mathbb{F}_2,B_T)$ of $(\pp^2,B)$ 
where $\mathbb{F}_2$ is the second Hirzebruch surface. This toric model induces an embedding $j_0\colon \mathbb{F}_2\setminus B_T \simeq \mathbb{G}_m^2 \hookrightarrow \pp^2\setminus B$.

We start with the trivial family 
\begin{equation}\label{fam1}
(\mathcal{T},B_{\mathcal{T}})\rightarrow \mathbb{D}, 
\end{equation}
where all the fibers are isomorphic
to the second Hirzebruch surface plus its toric boundary. 
In other words, we have 
\[
(\mathcal{T},B_{\mathcal{T}}) \simeq 
(\mathbb{F}_2,B_T)\times \mathbb{D}.
\]
Let $B_{\mathcal{T}}=S_{\mathcal{T}}+\Delta_{0,\mathcal{T}}+F_{0,\mathcal{T}}+F_{\infty,\mathcal{T}}$ where $S_{\mathcal{T}}\sim \Delta_0+2F\subset \mathbb{F}_2$ and the $f$'s are the fibers of the projection $\mathbb{F}_2\rightarrow \pp^1$. We blow-up the image of a section 
$s(\mathbb{D})\subset \mathcal{T}$ whose image is contained in 
$F_{0,\mathcal{T}}$ and $s(0)$ 
is the intersection of 
$F_{0,\mathcal{T},0}$ and $\Delta_{0,\mathcal{T}}$.
This leads to a family 
\begin{equation}\label{fam2}
(\mathcal{T}',B_{\mathcal{T}'})\rightarrow \mathbb{D} 
\end{equation}
where the general fiber is a blow-up of $\mathbb{F}_2$ along a general point of a fiber of $\mathbb{F}_2\rightarrow \pp^1$
and the central fiber is a toric surface. 
We may blow-down the strict transform of 
$F_{0,\mathcal{T}}$ in $\mathcal{T}'$ to obtain a family 
\begin{equation}\label{fam3}
(\mathcal{Y}^+,B_{\mathcal{Y}^+}) \rightarrow \mathbb{D}. 
\end{equation}
The general fiber of this family is 
isomorphic to $\mathbb{F}_1$ and 
$\mathcal{B}_{Y^+,t}$ has three components; 
a fiber, a section with negative self-intersection, and a section with positive self-intersection.
The central fiber of~\eqref{fam3} is $\mathbb{F}_3$ with its toric boundary.
By \cite[Corollary 3.23]{HTU17}, the family admits an anti-flip which leads to a degeneration 
\begin{equation}\label{fam4}
(\mathcal{Y},B_{\mathcal{Y}})\rightarrow \mathbb{D}. 
\end{equation}
The general fiber of~\eqref{fam4}
is isomorphic to the general fiber of~\eqref{fam3}.
The central fiber of~\eqref{fam4} is
isomorphic to the blow-up of 
$\mathbb{P}(1,1,4)$ at a torus invariant point (see for example \cite[Section 7]{UZ24}). 
The family $\mathcal{Y}\rightarrow \mathbb{D}$ admits a blow-down 
$\mathcal{X}\rightarrow \mathbb{D}$. 
We obtain a family of Calabi--Yau pairs 
\begin{equation}\label{fam5}
(\mathcal{X},\mathcal{B})\rightarrow \mathbb{D}.
\end{equation}

Below, we give an explicit description, modulo change of coordinates, of the family~\eqref{fam5} and the blow-up
$\mathcal{Y}\rightarrow \mathcal{X}$.
The family is given by 
\[
\mathcal{X}_t=\{ [x:y:z:w] \mid xy=z^2+tw = 0\} \subset \pp(1,1,1,2)
\]
for $t\in \mathbb{D}$.
It is straightforward to check that $\mathcal{X}_t\cong \pp^2$ for $t\neq 0$ and $\mathcal{X}_0\cong \pp(1,1,4)$.
Indeed, $\mathcal{X}_t$, for $t\neq 0$, is a smooth del Pezzo surface of Picard rank one. The family of curves $\mathcal{B}\to \mathbb{D}$, where 
\begin{equation}
\mathcal{B}_t=\{[x:y:z:w] \mid xyw^2+tz^2w^2=0\} \subset \mathcal{X}_t \subset \pp(1,1,1,2).
\end{equation}
When $t=0$, $\mathcal{B}_0$ corresponds to the toric boundary of $\pp(1,1,4)$, while for $t\neq 0$ it follows that 
\begin{equation*}
B_t=\{[x:y:z:w] \mid w^2((1+t)z^2+tw)=0\}
\end{equation*}
 which corresponds to the union of the line $L_t=\{w=0\}$ and the irreducible conic 
\begin{equation*}
Q_t=\{[x:y:z:w]\mid (1+t)z^2+tw=0\}.
\end{equation*} 
The intersection points of $L_t$ and $Q_t$ are $[1:0:0:0]$ and $[0:1:0:0]$ which specialize to the smooth toric invariant points of $\mathcal{X}_0$. 
The family~\eqref{fam4} described above is obtained by blowing up the section $s:\mathbb{D}\to \mathcal{X}$ such that $s(t)=[1:0:0:0]$, we obtain a $\qq$-Gorenstein degeneration $\mathbb{F}_1\rightsquigarrow Bl_p\pp(1,1,4)$.

By construction, the family~\eqref{fam5} admits the embedding of a product 
\[
\xymatrix{ 
\mathbb{D} \times \mathbb{G}_m^2 \ar@{^{(}->}[rr]\ar[rd]_-{\pi_1} & &  \mathcal{X}\setminus \mathcal{B} \ar[ld] \\
& \mathbb{D}.& 
}
\]
The restriction of the previous embedding to the general fiber recovers $j_0$
while the restriction to the central fiber 
is just the embedding of the open torus into the toric variety $\pp(1,1,4)$.
}
\end{example} 

We finish this section proposing some questions for further research. In~\cite{AdSFM24}, the first author, together with Alves da Silva and Figueroa, gave a partial classification of cluster type pairs $(\pp^3,B)$ for the projective space $\pp^3$. Following the philosophy of Example~\ref{ex:P^2-toric-model-1}, it is natural to expect a classification of cluster type degenerations of $\pp^3$. 

\begin{question}
Let $(\pp^3,B)$ be a cluster type pair. Is it possible to classify all the cluster type degenerations of this pair? 
\end{question}

Philosophically, degenerations of Fano varieties introduce deeper singularities. Theorem~\ref{no-deg-almost} hints that Fano surfaces with deeper singularities have fewer degenerations. This motivates the question of studying iterated degenerations of Fano varieties. For instance, using the work of Hacking and Prokhorov~\cite{HP10}, one can show that an iteration of (non-trivial)
degenerations of $\pp^2$ is finite, indeed, it has at most three steps.
This motivates the following question. 

\begin{question}
Are iterations of (non-trivial) $\qq$-Gorenstein klt degenerations of $\pp^3$ finite? Are maximal degenerations toric?
\end{question} 

\bibliographystyle{habbvr}
\bibliography{bib}

\end{document}